\documentclass[11pt, a4paper]{article}
\usepackage{amsmath,amssymb}
\usepackage[utf8]{inputenc}
\usepackage[colorlinks=true, citecolor=blue]{hyperref}
\usepackage{amsthm}
\usepackage{enumerate}
\usepackage{fullpage}
\usepackage{color}
\usepackage{graphicx}
\usepackage{float} 
\usepackage{subfigure}
\usepackage{mathtools}
\usepackage{enumitem}
\usepackage{tikz}
\usetikzlibrary{decorations.markings}
\usetikzlibrary{patterns}
\numberwithin{equation}{section}
\usepackage{float}
\usepackage[capitalize,nameinlink,noabbrev]{cleveref}
\usepackage{array}
\usepackage{calc}
\newtheorem{theorem}{Theorem}
\newtheorem{lemma}[theorem]{Lemma}

\newtheorem{definition}[theorem]{Definition}

\usepackage{pifont}
\usepackage{mathtools}
\usepackage{mathrsfs}
\usepackage{array}
\usepackage{verbatim}
\usepackage[affil-it]{authblk}
\usepackage{esvect}

\DeclareMathOperator{\bigO}{O}

\title{\bf{A Combinatorial Insight to the Riemann Boundary Value Problem in Lattice Walks}}
\author{Ruijie Xu\thanks{\href{mailto:xuruijie@bimsa.cn}{xuruijie@bimsa.cn}}}

\affil{\small Beijing Institute of Mathematical Science and Applications (BIMSA)}
\date{ } 

\begin{document}

\maketitle
\begin{abstract}
The enumeration of quarter-plane lattice walks with small steps is a classical problem in combinatorics. An effective approach is the kernel method, where the solution is derived by positive term extraction. Alternatively, one may reduce the lattice walk problem to a Carleman-type Riemann boundary value problem (RBVP) and solve it via analytic method. In the RBVP framework, two parameters govern the solution: the index $\chi$ the conformal gluing function $w(x)$. 

In this paper, we propose a combinatorial insight into the RBVP approach. We show that the index corresponds to the canonical factorization in the kernel method. The conformal gluing function can be viewed as a mapping that enables the application of positive term extraction. The combinatorial insight of RBVP establishes a unifying link between the kernel method, the RBVP approach and the Tutte's invariants method.
\end{abstract}
\tableofcontents
\section{Introduction}
Quarter-plane lattice walk models describe walks with fixed step sets (allowed steps in terms of lengths and directions), ending on the grids and restricted in the first quadrant. By solving a lattice walk problem, we mean finding an explicit expression of the number of configurations of $n$-step walks ending on an arbitrary vertex. This is equivalent to determining the generating function of the number of configurations. Lattice walk enumeration is a classical combinatorial problem with connections to probability theory \cite{malyshev1972analytical}, mathematical physics \cite{brak2005directed}, integrable systems \cite{Tong2021}, and representation theory \cite{postnova2021counting}. 

Quarter-plane lattice walks with small steps have been extensively studied. In \cite{bousquet2010walks}, Bousquet-M\'elou and Mishna classified all $256$ possible walks into $79$ distinct non-trivial two-dimensional models. Among these, $23$ models are associated with finite symmetry groups: $16$ models with $D_2$ groups (dihedral $2$), five models with $D_3$ groups and two with $D_4$ groups. 

Various methods exist for solving lattice walk problems. In \cite{bousquet2010walks} and \cite{bousquet2016elementary}, the kernel method was introduced to solve all models with finite symmetry groups. In \cite{Melczer2014}, the iterated kernel method was used to derive asymptotic formulas and explicit generating functions for certain singular models.

Prior to combinatorial studies, probabilists developed analytical approaches for quarter-plane random walks. In \cite{malyshev1972analytical}, random walks were reduced to a Riemann boundary value problem (RBVP) with Carleman shift \cite{Litvinchuk2000}. This framework was later adapted for lattice walk enumeration in \cite{raschel2012counting}.

Other works on quarter-plane walks \cite{bernardi2007bijective,fayolle2017random,bousquet2005walks,kauers2008quasi,mishna2009classifying,raschel2020counting,dreyfus2018nature,kurkova2012functions} employ diverse methods from combinatorics, probability theory, algebraic geometry, complex analysis, and computer algebra.

\section{Motivation}\label{motivation}
\subsection{From combinatorics to Riemann boundary value problems}
Various approaches yield distinct solution forms, yet the equivalence of these solutions is not immediately apparent. For example, consider the simple lattice walk (walk with allowed steps $\{\uparrow,\downarrow,\leftarrow,\rightarrow\}$), starting form the origin $(0,0)$, restricted in the quarter-plane. Using bijections and inclusion-exclusion, the solution is given by \cite{guy1992lattice},
\begin{align}
    |P^{++}_n((0,0)\to(i,j))|=\binom{n}{\frac{n+i-j}{2}}\binom{n+2}{\frac{n-i-j}{2}}-\binom{n+2}{\frac{n+i-j}{2}+1}\binom{n}{\frac{n-i-j}{2}-1},
\end{align}
where $|P^{++}_n((0,0)\to(i,j))|$ denotes the number of $n$-step walks from $(0,0)$ to $(i,j)$. 

Using the algebraic kernel method \cite{bousquet2005walks}, the solution takes the form,
\begin{align}
    Q(x,y,t)=[x^>y^>]\frac{xy-x/y-y/x+1/x/y}{K(x,y)}\label{sol alg ker}.
\end{align}
where $Q(x,y,t)$ (abbreviated as $Q(x,y)$) is the generating function,
\begin{align}
    Q(x,y)=\sum_{i,j,k}q_{ijn}x^iy^jt^n,
\end{align}
with $q_{ijn}$ representing the number of $n$-step paths from $(0,0)$ to $(i,j)$ and $i,j,n\geq 0$. $[x^i]f(x)$ denotes the coefficient of $x^i$ term in $f(x)$. and $[x^>]f(x)$ extracts terms with positive powers of $x$. The kernek $K(x,y)$ is defined as,
\begin{align}
    K(x,y)=1-t(x+1/x+y+1/y).
\end{align}

The obstinate kernel method \cite{prodinger2004kernel,banderier2002basic} first solves the boundary values $Q(x,0)$ and $Q(0,y)$ instead of $Q(x,y)$ directly. The solution reads,
\begin{align}
    Q(x,0)=\frac{1}{tx}[x^>]Y_0(x)(x-1/x)\label{simple result},
\end{align}
where $Y_0(x),Y_1(x)$ are roots of $K(x,Y)=0$. $Y_0(x,t)$ is treated as a formal power series in $t$. Both kernel methods utilize positive degree term extraction ($[x^>]$), a technique common in solving discrete difference equations \cite{buchacher2022orbit}.

It is straightforward to obtain the solution form of bijections and inclusion-exclusion from the algebraic kernel method. For the simple lattice walk, $|P^{++}_n((0,0)\to(i,j))|=q_{i,j,n}=[x^iy^jt^n]Q(x,y,t)$ for $Q(x,y,t)$ in \eqref{sol alg ker}. This can be checked by direct calculations.

The connection between algebraic and obstinate kernel methods arises from:
\begin{equation}
\frac{1}{K(x,y)}=\frac{1}{a(x,t)(Y_0-Y_1)}\left( \sum_{i\geq0}\left(\frac{Y_0(x)}{y}\right)^i+\sum_{i>0}\left(\frac{y}{Y_1(x)}\right)^i\right), \label{canonical}
\end{equation}
and
\begin{align}
    [x^>]f(x,Y_0(x))=[x^>]f(x,y)-[x^>]\Big(K(x,y)[y^>]\frac{f(x,y)-f(x,1/y)}{K(x,y)}\Big)\label{trans 2}.
\end{align}
Here, $Y_0(x)$ and $Y_1(x)$ are two roots of $K(x,y)=0$\footnote{we always choose $Y_0$ as the root whose series expansion is a formal power series of $t$.}. \eqref{trans 2} and \eqref{canonical} link $[x^>y^>]\frac{1}{K(x,y)}$ to $[x^>]f(x,Y_0(x))$. The number of configurations (results of bijection and inclusion-exclusion) can also be calculated from the solution of the obstinate kernel method directly via the techniques in \cite{banderier2019}, a variant of Lagrange-B\"urmann inversion formula.

Unlike the combinatorial approaches, the RBVP methods yields a solutions in integral form,
\begin{align}
    c(x,t)Q(x,0,t)=\frac{1}{2\pi i}\oint_{X_0[y_1,y_2]}sY_0(s)\frac{\partial_s w(s,t)}{w(s,t)-w(x,t)}ds\label{RBVP walk}.
\end{align}
Here, $c(x,t)$ is a known coefficient. $X_0[y_1,y_2]$ denotes a contour on the $x$-plane. $X_0(y)$ is the roots of the kernel $K(X,y)=0$. If we write $X_0(y)=\frac{-b(y)-\sqrt{\Delta(y)}}{2a(y)}$, $y_1,y_2$ correspond to ``small root"s of $\Delta(y)$\footnote{In the combinatorial insight, 'small roots' are formal power series in $t$. In the analytic insight, they lie inside the unit circle for small fixed $t$.}.  $w(x,t)$ is a conformal mapping with additional gluing properties, called conformal gluing function. It is originally expressed via Weierstrass $\wp$ functions \cite{fayolle2017random} and later received algebraic expressions for models associated with finite groups \cite{raschel2012counting}. This integral form derives from solving Riemann boundary value problems (RBVPs) with Carleman shifts \cite{Litvinchuk2000}. Here is the definition:
\begin{definition}
 Find a function $\Phi^+$ holomorphic in some domain $L^+$, the limiting value of which on the closed contour $L$ are continuous and satisfies the equation,
\begin{align}
\Phi^+(\alpha(s))=G(s)\Phi^+(s)+g(s), \qquad s\in L\label{BVP1},
\end{align}
where
\begin{enumerate}
\item For any $s \in L$, $g$ and $G$ are H\"older continuous on $L$ and $G(s)\neq 0$.
\item The function $\alpha(s)$ is a one-to-one mapping of $ L\to L$ such that $\alpha(\alpha(s))=s$ and $\alpha$ is H\"older continuous.
\end{enumerate}
\end{definition}
A function $f$ is said to be H\"older continuous \cite{omnes1958solution} if for any two points $s_1,s_2$ on $L$, one can find positive constants $A$ and $\mu$, such that,
\begin{align}
    |f(s_1)-f(s_2)|\leq A|s_1-s_2|^{\mu}.
\end{align}
In addition to the basic definitions, we let,
\begin{align}
    \begin{split}
        &G(s)G(\alpha(s))=1\\
        &G(\alpha(s))g(s)+g(\alpha(s))=0,
    \end{split}
\end{align}
to avoid the trivial cases that we can solve \eqref{BVP1} directly by substitution,
\begin{align}
    \Phi^+(s)=\frac{G(\alpha(s))g(s)+g(\alpha(s))}{1-G(s)G(\alpha(s))}.
\end{align}
As shown in \cite{Litvinchuk2000}, there exists a function $z=w(x)$, which is holomorphic in $L^+$ except one point $x_0$ and satisfies the following gluing property,
\begin{align}
w(\alpha(s))-w(s)=0\qquad s\in L.\label{conformal}
\end{align}
It is a conformal mapping from the domain $L^+$ in $x$-plane to $z$-plane with cut $\lfloor ab\rfloor$, where $s\in L$ maps to $u\in [ab]$. The inverse map is $x=w^{-1}(z)$.

Let $\chi=-\frac{Arg(G(w^{-1}(u)))|_{\lfloor ab\rfloor}}{2\pi}$. $Arg$ denotes the argument. $\chi$ is called the index. $\lfloor ab\rfloor$ refers to the lower cut of $[ab]$.  If $\chi= -1$, 
the solution reads,
\begin{align}
    \Phi^{+}(x)=\frac{X(w(x))}{2\pi i}\oint_{\lfloor ab\rfloor}\frac{g(w^{-1}(u))}{X^+(u)(u-w(x))}du=-\frac{X(w(x))}{2\pi i}\oint_{L_d}\frac{\partial_s w(s)g(s)}{X^+(w(s))(w(s)-w(x))}ds,
\end{align}
where
\begin{align}\label{index rbvp}
    \begin{cases}
        & X(z)=(z-b)^{-\chi}e^{\Gamma(z)}\\
        &\Gamma(z)=\frac{1}{2\pi i}\int_{\lfloor ab\rfloor}\frac{log(G(w^{-1}(u)))}{u-z}du.
    \end{cases}
\end{align}
 $L_d$ is the lower half of $L$ (the preimage of $\lfloor ab\rfloor$). If $\chi<0$, the homogeneous problem ($g(s)=0$) has no solution. The nonhomogeneous problem is solvable if $|-\chi-1|$ extra conditions are satisfied,
\begin{align}
    \frac{1}{2\pi i}\int_{\lfloor ab \rfloor}\frac{g(w^{-1}(u))u^{k-1}}{X^+(u)}du=0\qquad k=1,2,\dots -\chi-1.\label{RBVP requirement}
\end{align}
If $\chi\geq 0$, \eqref{BVP1} has $1+\chi$ linearly independent solutions,
\begin{align}
     \Phi^{+}(x)=\frac{X(w(x))}{2\pi i}\oint_{\lfloor ab\rfloor}\frac{g(w^{-1}(u))}{X^+(u)(u-w(x))}du+X(w(x))P_{\chi}(x).\label{RBVP solution 2}
\end{align}
where $P_{\chi}(x)$ is a polynomial of degree $\chi$.

\eqref{RBVP walk} corresponds to $G(z)=1$ and $\chi=0$\footnote{The constant term of $c(x,t)Q(x,0,t)$ vanishes by the definition of lattice walk, forcing $P_0(x)=0$}. But for three-quadrant walks \cite{raschel2018walks} or weighted walks \cite{fayolle2017random}, $\chi\neq 0$ in general.

A natural question arises: Does \eqref{RBVP walk} equal \eqref{simple result} for simple lattice walks? Furthermore, the RBVP introduces $\chi$  (governing solution multiplicity)  and $w(x)$ (defining the integral contour). $\chi$ characterize the number of linearly independent solutions and $w(x)$ characterize the integral. While these are central to RBVP, combinatorial approaches \cite{mishna2009two,bousquet2005walks,beaton2019quarter,beaton2018exact,bousquet2016elementary} implicitly obscure them.

To show where and how $\chi$ emerges in the obstinate kernel method, we extend the problem to quarter-plane walks with boundary interactions \cite{beaton2018exact,beaton2019quarter,xu2022interacting}. We show that the combinatorial correspondence of $\chi$ characterizes the solvability of certain models with interactions.

To show where and how $w(z)$ emerges in the obstinate kernel method, we will compare the combinatorial and analytic insights of the lattice walk problems, and find a combinatorial version of the conformal gluing function.

\subsection{Organization of this paper}
The objective of this paper is to introduce the index $\chi$ and the conformal gluing function $z=w(x)$ into the kernel method. We first review the idea of the kernel method and the Riemann boundary value problem in \cref{uniform} and demonstrating how the index can be naturally defined in the combinatorial insight. Due to the differences between the obstinate kernel method and the RBVP approach, index $\chi$ in the kernel method may not be unique, leading to distinct phenomena. We use an example in lattice walk to show this in \cref{5 and 15}.  

From \cref{combinatorial cfg} to \cref{tutte}, we show that the conformal gluing function is introduced into the obstinate kernel method differently from the RBVP framework. We prove that under a M\"obius transform, the conformal gluing function from the RBVP can be adapted to the kernel method as a combinatorial conformal mapping. This conformal mapping can be treated as a generalization of the week invariant in Tutte's invariant method \cite{raschel2020counting}.

\section{Definition and Notations}\label{def}
\subsection{Notations for algebra}
Suppose $f(x)$ is a series with both $x$ and $1/x$ terms. We denote $[x^i]f(x)$ as the coefficient of $x^i$ in $f(x)$. $[x^>]f(x)$, $[x^<]f(x)$ and $[x^\geq]f(x)$ as the terms of positive, negative and non-negative powers of $x$ in $f(x)$.

For a ring or a field $\mathbb{K}$, we denote
\begin{enumerate}
\item $\mathbb{K}[t]$ as the set of polynomials in $t$ with coefficients in $\mathbb{K}$;
\item $\mathbb{K}[t,\frac{1}{t}]$ as the set of polynomials in $t$ and $\frac{1}{t}$ with coefficients in $\mathbb{K}$.
\item $\mathbb{K}[[t]]$ as the set of formal power series in $t$ with coefficients in $\mathbb{K}$;
\item $\mathbb{K}((t))$ as the set of Laurent series in $t$ with coefficients in $\mathbb{K}$;
\item $\mathbb{K}(t)$ as the set of rational functions in $t$ with coefficients in $\mathbb{K}$.
\end{enumerate}
The notations extend to multiple variables. For example $\mathbb{R}((x))((t))$ refers to the set of Laurent series in $t$ with coefficients in the set of formal power series of $x$ with real coefficients. Also notice that multivariate Larent expansion are defined iteratedly \cite{melczer2021invitation} and $\mathbb{R}((x))((t))\neq \mathbb{R}((t))((x))$.

We consider the following set of functions,
    \begin{enumerate}
\item \textbf{Algebraic:} $f(x)$ is algebraic over $x$ if there exists a polynomial equation $P(f,x)=0$ with coefficients in $\mathbb{C}(x)$.
\item \textbf{D-finite:} $f(x)$ is D-finite (short for differentiably finite, also called holonomic) over $x$ if there exists a linear equation $L(f,f',\dots f^{(n)})=0$ with coefficients in $\mathbb{C}(x)$.
\item \textbf{D-algebraic:} $f(x)$ is D-algebraic (differentiably algebraic, also called hyper-algebraic) over $x$ if there exists a polynomial equation $P(f,f',\dots f^{(n)})=0$ with coefficients in $\mathbb{C}(x)$.
\item \textbf{Hyper-transcendental} $f(x)$ does not satisfy any of the previous three conditions.
    \end{enumerate}
\begin{figure}[ht!]
\centering
\includegraphics[scale=0.25]{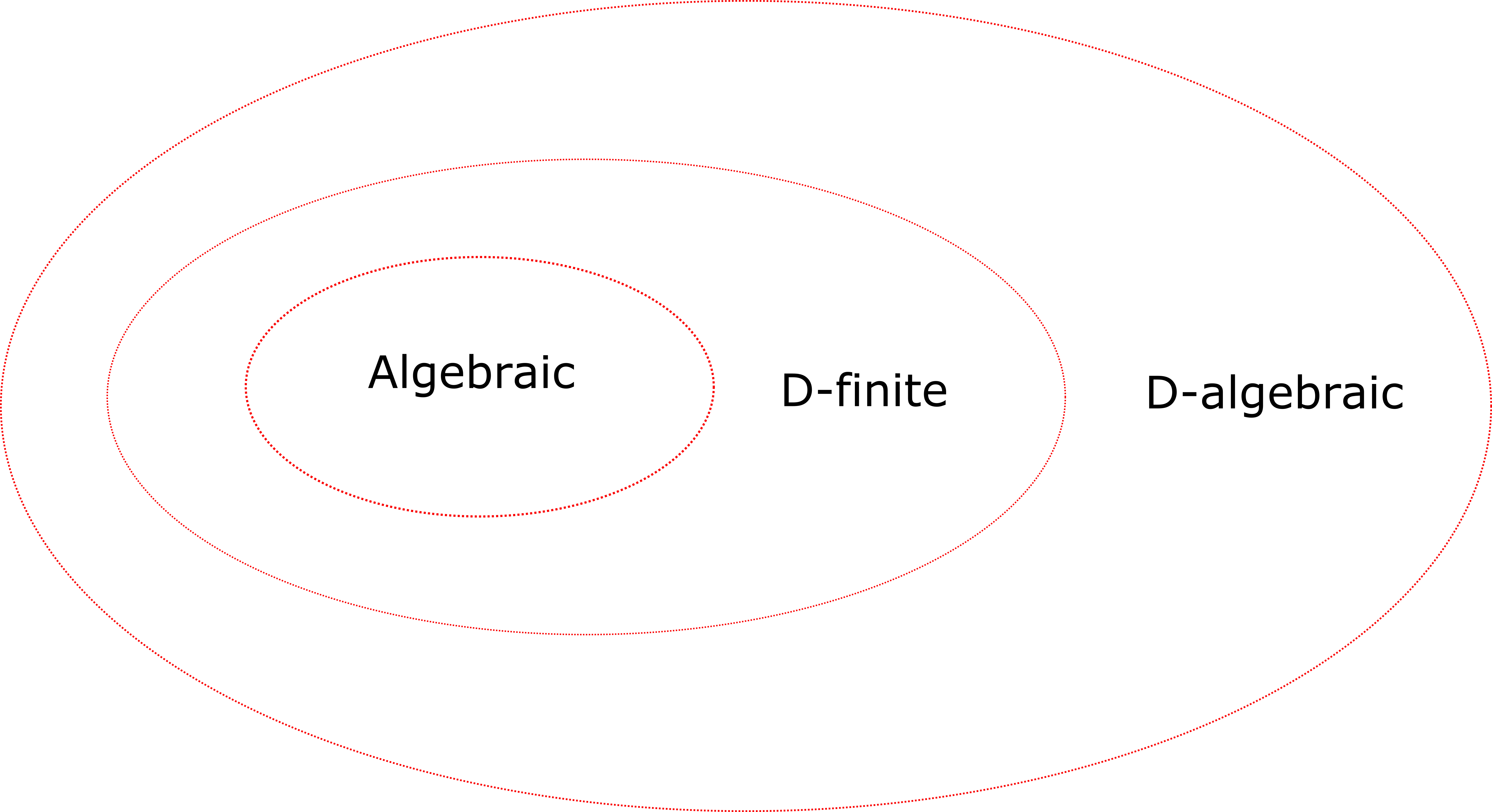}
\caption{The relation between algebraic, D-finite and D-algebraic}
\label{fig The relation between algebraic, D-finite and D-algebraic}
\end{figure}
\subsection{Notations for lattice walk models}
To include the index $\chi$ in a lattice walk model, we introduce the boundary weights. We denote $q_{n,k,l,h,v,u}$ as the number of paths of length $n$ that start at $(0,0)$, end at $(k,l)$, and visit vertices on the horizontal boundary (except the origin) $h$ times, on the vertical boundary (except the origin) $v$ times and the origin $u$ times. The generating function is
\begin{align}
Q(x,y,t;a,b,c)\equiv Q(x,y)=\sum_{n}t^n\sum_{k,l,h,v,u}q_{n,k,l,h,v,u}x^k y^l a^h b^v c^u &\equiv \sum_{n}t^n Q_n(x,y) \notag \\ &\equiv\sum_{i,j}Q_{i,j}x^iy^j\label{function}.
\end{align}
For convenience, we always omit $t$ in the notation.

We define line-boundary terms:
\begin{equation}
[y^i]Q(x,y,t;a,b,c)\equiv Q_{-,i}(x)=\sum_{n}t^n\sum_{k,h,v,u}q_{n,k,i,h,v,u}x^k a^h b^v c^u.
\end{equation}
This is the generating function of walks ending on the line $y=i$. $Q_{i,-}(y)$ is defined similarly. For walks ending on the line $y=0$ ($x=0$), we prefer a more convenient expression $Q(x,0)$ or $Q(0,y)$.

For walks without interactions, just set $a,b,c$ equals $1$ and we get the corresponding $Q(0,y)$, $Q(x,0)$ and $Q(x,y)$.

\subsection{Derivation of the functional equation}\label{Model}
Consider a walk starting from the origin with allowed steps $\mathcal{S} \subseteq \{-1,0,1\}\times \{-1,0,1\}\setminus \{0\}\times \{0\}$. The \emph{step generator} $S$ is
\begin{align}
S(x,y) = \sum_{(i,j)\in\mathcal{S}}x^i y^j.
\end{align}
It can be written as
\begin{align}
S(x,y)=A_{-1}(x)\frac{1}{y}+A_0(x)+A_1(x)y=B_{-1}(y)\frac{1}{x}+B_0(y)+B_{1}(y)x\label{S2}
\end{align}
$|S|$ refers to the number of allowed steps.
\begin{theorem}\label{theorem1}\cite{xu2022interacting}
For a lattice walk restricted to the quarter-plane, starting from the origin, with weight `$a$'  associated with steps ending on the $x$-axis (except the origin), weight `$b$' associated with steps ending on $y$-axis (except the origin) and weight `$c$' associated with steps ending at the origin, the generating function $Q(x,y)$ satisfies the following functional equation:
\begin{multline}
K(x,y)Q(x,y)=\frac1c + \frac1a\Big(a-1 - taA_{-1}(x)\frac{1}{y}\Big)Q(x,0) + \frac1b\Big(b - 1 - tbB_{-1}(y)\frac{1}{x}\Big)Q(0,y) \\
+\left(\frac{1}{abc}(ac+bc-ab-abc)+\frac{t[x^{-1}y^{-1}]S(x,y)}{xy}\right)Q(0,0)\label{functional}
\end{multline}
where $K(x,y)=1-tS(x,y)$. $K(x,y)$ is called the kernel of the walk. 
\end{theorem}
We refer to \cite{xu2022interacting} for detailed proof. We will directly use \eqref{functional} in this paper.
\section{A Uniform Functional Equation for Different Methods}\label{uniform}
\subsection{A sketch of the kernel method approach for simple lattice walk }\label{kernel method}
To show how we obtain a functional equation analogous to the Riemann boundary value problem \eqref{BVP1}, we consider the simple lattice walk in the quarter plane. $S(x,y)=x+1/x+y+1/y$ for this model. By \eqref{functional} with $a=b=c=1$, The functional equation reads
\begin{align}
    K(x,y)Q(x,y)=1-\frac{t}{y}Q(x,0)-\frac{t}{x}Q(0,y)\label{func model1}
\end{align}
where
\begin{align}
    K(x,y)=1-t(x+1/x+y+1/y)\label{kernel}
\end{align}
It has two roots,
\begin{align}
\begin{split}
&Y_0(x)=-\frac{\sqrt{\left(t \left(x+\frac{1}{x}\right)-1\right)^2-4 t^2}}{2 t}+\frac{1}{2 t}-\frac{1}{2} \left(x+\frac{1}{x}\right)=t+ \left(x+\frac{1}{x}\right)t^2+O(t^3)\\
&Y_1(x)=\frac{\sqrt{\left(t \left(x+\frac{1}{x}\right)-1\right)^2-4 t^2}}{2 t}+\frac{1}{2 t}-\frac{1}{2} \left(x+\frac{1}{x}\right)=\frac{1}{t}+\frac{1+x^2}{x}-t-\frac{1+x^2}{x}t^2+O(t^3).
\end{split}
\end{align}
The kernel $K(x,y)$ is invariant under the following two involutions,
\begin{align}
    \phi:(x,y) \to (1/x,y) \qquad\text{and}\qquad  \psi:(x,y)\to (x,1/y),
\end{align}
and these two involution general the group,
\begin{align}
(x,y), (1/x,y),(x,1/y),(1/x,1/y)
\end{align}
It is a $D_2$ group (Dihedral group).

The idea of the obstinate kernel method, is that if we substitute the formal power series (in $t$) roots of $xyK(x,y)=0$ into \eqref{func model1}, the left hand-side is $0$ and is invariant under the involution $\phi: (x,Y_0(x))\to (\bar{x},Y_0(\bar{x}))=(\bar{x},Y_0(x))$. Since $Y_0(x)$ is unchanged under the involution, we can eliminate the unknown function $Q(0,Y_0)$. The choice of formal series root here is important since substitution with $1/t$ terms is not well defined in general \cite{xin2004ring}.

By substitution and applying the transformation $x\to \bar{x}$, we obtain two equations,
\begin{align}
    &xY_0-txQ(x,0)-tY_0Q(0,Y_0)=0\label{func simple}\\
    &\bar{x}Y_0-t\bar{x}Q(\bar{x},0)-tY_0Q(0,Y_0)=0
\end{align}
Here $Y_0\equiv Y_0(x)=Y_0(\bar{x})$. A linear combination eliminates $Q(0,Y_0)$,
\begin{align}
    (x-\bar{x})Y_0=txQ(x,0)-t\bar{x}Q(\bar{x},0)\label{sep}.
\end{align}
In \eqref{sep}, the unknown function $F(x,0)\in \mathbb{C}[x][[t]]$ and the unknown function $F(\bar{x},0)\mathbb{C}[\bar{x}][[t]]$ are separated. We can take $[x^>]$ and $[x^<]$ of \eqref{sep},
\begin{align}
    \begin{split}
        &txQ(x,0)=[x^>](x-\bar{x})Y_0\\
        &t\bar{x}Q(\bar{x},0)=[x^<](x-\bar{x})Y_0.\label{solution kernel}
    \end{split}
\end{align}
The solution of $F(x,0)$ is expressed by $x,Y_0$ and $[x^>]$ operator. $Q(0,y)$ can be obtained by consider the substitute of $X_0(y)$ by a similar way.
\begin{align}
\begin{split}
&X_0(y)=-\frac{\sqrt{\left(t \left(y^2+1\right)-y\right)^2-4 t^2y^2}}{2y t}+\frac{1}{2 t}-\frac{1}{2} \left(y+\frac{1}{y}\right)\\
&X_1(y)=\frac{\sqrt{\left(t \left(y^2+1\right)-y\right)^2-4 t^2y^2}}{2y t}+\frac{1}{2 t}-\frac{1}{2} \left(y+\frac{1}{y}\right).\label{X0y}
\end{split}
\end{align}
\subsection{A sketch of the Riemann boundary value problem for simple lattice walk}
The idea of the RBVP approach is to treat \eqref{func model1} as an equation of meromorphic functions. In the analytic insight \cite{fayolle2017random,raschel2012counting}, $t$ is treated as a fixed value with $t<1/|S|$. $|S|$ is the number of allowed steps. $xyK(x,y)=0$ is considered as an algebraic curve. The functional equation \eqref{func model1} is defined as formal series in the domain $\{|x|<1\}\cap\{|y|<1\}\cap\{|z|<1\}$. $Q(x,0),Q(0,y)$ are formal series expressions of meromorphic functions in local charts and can be analytic continued to the whole curve.

Denote the polynomial within the square roots of $X_0(y),X_1(y)$ as $\Delta$. $\Delta(x,t)$ has four roots, $x_1,x_2,x_3,x_4$ and $|x_1|<|x_2|<1<|x_3|<|x_4|$. $\Delta(y,t)$ has four roots $|y_1|<|y_2|<1<|y_3|<|y_4|$\footnote{For some other models, there are three roots $|x_1|<|x_2|<1<|x_3|$ with $x_4\to \infty$, or $|y_1|<|y_2|<1<|y_3|$ with $x_4\to \infty$} . 

On the $x$-plane, the Jordan curve $X_0([y_1y_2])$ and $X_1([y_1y_2])$ coincide. The involution $\phi$ is considered as the Galois transformation $(X_0(y),y)\to (X_1(y),y)$. For any point $s\in X_0(y_1y_2)$ and the Galois transformation, we obtain,
\begin{align}
    &sY_0(s)-tsQ(s,0)-tY_0(s)Q(0,Y_0(s))=0\\
    &\phi(s)Y_0(\phi(s))-t\phi(s)Q(\phi(s),0)-tY_0(\phi(s))Q(0,Y_0(\phi(s)))=0
\end{align}
$\phi$ maps $s$ to its conjugate point $\phi(s)$, which is still on the curve $X_0(y_1y_2)$. Thus both equations hold on the curve $X_0(y_1y_2)$ and a linear combination gives,
\begin{align}
sY_0(s)-\phi(s)Y_0(\phi(s))=tsQ(s,0)-t\phi(s)Q(\phi(s),0)\label{sep 2}.
\end{align}
If we denote $\Phi(s)=tsQ(s,0)$, $\alpha=\phi$, \eqref{sep 2} is exactly the Riemann boundary value problem with Carleman shift \eqref{BVP1} with $G(s)=1$. $\chi=0$. There is a conformal gluing function $z=w(x)$, which maps the domain inside $X_0(y_1y_2)$ to a slit plane with cut $[ab]$. If $s\in X_0([y_1y_2])_d$ (lower part of $s\in X_0([y_1y_2])$) is mapped to $u^-$ on $\lfloor ab\rfloor$, $\phi(s)$ is mapped $u^+$ on $\lceil ab\rceil$ ($\lfloor~\rfloor$ and $\lceil~\rceil$ refer to the lower and upper cut). We denote $w^{-1}(z)$ as $x(z)$ to avoid too many exponents in the notation. After conformal mapping, we have,
\begin{align}
    (x(u^-))Y_0(x(u^-))-(x(u^+))Y_0(x(u^+)=tw^{-1}(u^-)Q(x(u^-),0)-tx(u^+)Q(x(u^+),0).
\end{align}

The condition
\begin{align}
    \frac{1}{2\pi i}\int_{\lfloor ab \rfloor}\frac{(x(u^+)Y_0(x(u^+)-x(u^-)Y_0(x(u^-))}{u}du=0.
\end{align}
is fulfilled. Then, the solution reads,
\begin{align}
    tx(z)Q(x(z),0)=\frac{1}{2\pi i}\int_{\lfloor ab \rfloor}\frac{(x(u^+)Y_0(x(u^+)-x(u^-)Y_0(x(u^-))}{u-z}du=0.\label{RBVP solution 1}
\end{align} 
By change of variable, we get \eqref{RBVP walk}.
\subsection{A generic form for the combinatorial analog of RBVP}
Comparing \eqref{sep} and \eqref{sep 2}, the forms of the equations differ by $x\to X_0(s)$. However, it is not rigorous to claim that they are equivalent since the definitions of the functions are different. Both the obstinate kernel method and the Riemann boundary value problem arrive at an equation in this form and we consider \eqref{sep} a combinatorial analog of Riemann boundary value problem with Carleman shift.

To be consistent with the general form of RBVP \eqref{BVP1}, we define the combinatorial Riemann boundary value problem (cRBVP) as,
\begin{align}
   Q(1/x,0)+G(x,t)Q(x,0)=g(x,t),\label{general}
\end{align}
$Q(1/x,0)$ and $Q(x,0)$ are defined as formal series of $t$. We still consider the non-trivial case,
\begin{align}
\begin{split}
    &G(x,t)G(1/x,t)=1\\
    &G(1/x,t)g(x,t)+g(1/x,t)=0  
\end{split}
\end{align}
From the calculation of the simple lattice walk, we note that $G(x,t)$ and $g(x,t)$ are rational functions of $x,t, Y_0,Y_1$. Normally, they can be denoted as a function $p+q\sqrt{\Delta}$ where $p,q,\Delta$ are rational in $x,t$. We may further consider them in $\mathbb{C}((x))((t))$ or $C((1/x))((t))$ or even $\mathbb{C}((x,1/x))((t))$ in the following section.

\section{The Arise of Index}\label{ss}
\subsection{Canonical factorization}
To understand the index, we first consider factorizing $G(x,t)$. Here we apply Gessel's canonical factorization theorem,
\begin{theorem}[Theorem 4.1 in \cite{gessel1980factorization}]\label{canonical factor of f}
Let $f$ be a function in $\mathbb{C}[[x,\frac{t}{x}]]$ with constant term $1$. Then $f$ has a unique decomposition $f=f_-f_0f_+$, where $f_-$ is of the form $1+\sum_{i>0,j>0}a_{ij}t^ix^{-j}$, $f_0$ is of the form $1+\sum_{i>0}a_i t^i$, and $f_+$ is of the form $1+\sum_{i\geq 0,j>0}a_{ij}t^ix^j$.
\end{theorem}
\begin{proof}
Let $\log(f)=\sum_{ij}b_{ij}t^ix^j$. Then $f_-=e^{\sum_{i\geq 0,j<0}b_{ij}t^ix^j}$. $f_0=e^{\sum_{i\geq0}^{\infty}b_{i0}t^i}$, and $f_+=e^{\sum_{i\geq 0,j>0}b_{ij}t^ix^j}$. These three are all well defined series in $\mathbb{C}[[x,\frac{t}{x}]]$. Expand the exponential functions and the proof is finished.
\end{proof}
Inspired by \cref{canonical factor of f}, if $f$ has a Laurent expansion in $\mathbb{C}((x))((t))$, $\log(f(x,t))$ can always be expanded as $[x^<]\log f+[x^0]\log f+[x^>]\log f+k\log x+l\log t$ in its convergent domain. Suppose $f$ is expanded as a series of $t$
\begin{align}
f=a_{l}(x)t^l+a_{1}(x)t^{l+1}+a_{2}(x)t^{l+2}+\dots.
\end{align}
$l=\min\{i : a_i \neq 0\}$. If we expand $a_i(x)$ in $\mathbb{C}((x))$, we write $a_l$ as
\begin{align}
a_l=b_{k}x^k+b_{k+1}x^{k+1}+\dots=b_{k}x^k(1+r(x)).\label{cal k}
\end{align}
where $b_k$ is the coefficient of the smallest degree of $x$ in $a_l$ (can be negative) and $r(x)$ is a formal series of $x$ without constant terms. We take the $x^kt^l$ out and write $f$ as
\begin{align}
    f=(b_kx^kt^l)\times(1+r(x)+tg(x,t)).
\end{align}
$g(x,t)$ is a formal series of $t$. Then, by the Taylor expansion of $\log(1+x)$,
\begin{align}
\begin{split}
    \log(f(x,t))&=\log(b_kx^kt^i)+\log(1+r(x)+g(x,t))\\
    &=\log(b_kx^kt^l)+\log(1+r(x))+\log\left(1+\frac{g(x,t)}{1+r(x)}\right)\\
    &=\log(b_kx^kt^l)+\log(1+r(x))+1+\frac{g(x,t)}{1+r(x)}+\frac{1}{2}\left(\frac{g(x,t)}{1+r(x)}\right)^2+\dots
    \end{split}
\end{align} 
We can further expand $\log(1+r(x))$ as a Laurent series of $x$ and $\frac{g(x,t)}{1+r(x)}$ as formal series of $t$, then $log(f(x,t))$ is a formal series of $t$ and we can take $[x^>],[x^0],[x^<]$ terms of it.

We may also expand $f$ and $a_i(x)$ in $\mathbb{C}((\frac{1}{x}))((t))$. The only difference is 
\begin{align}
a_l=b_{-k}x^{-k}+b_{-(k+1)}x^{-(k+1)}+\dots=b_{-k}x^{-k}(1+r(1/x))
\end{align}
The results is not unique. It can be functions in $\mathbb{C}((x))[[t]]$, $\mathbb{C}((1/x))[[t]]$ or $\mathbb{C}((x,1/x))[[t]]$ depending on the convergent domain we choose. If we fixed the convergent domain, then the factorization is unique. In lattice walk problem, we face a function in the form $\log(p+q\sqrt{\Delta})$. It is meromorphic in some annulus.

\subsection{An index criteria for linear independent solution}\label{chi section}
It is straightforward to consider the factorization as $G(x,t)=t^lx^{-\chi}e^{G_-}e^{G_0}e^{G_+}$ and $\chi$ is the index. By direct calculation, we have,
\begin{align}
   \frac{x^\chi}{e^{G_-}}Q(1/x,0)+t^le^{G_+}e^{G_0}Q(x,0)=\frac{g(x,t)}{e^{G_-}}.\label{general 2}
\end{align}

This is the suitable form for taking $[x^{\geq}]$ and $[x^<]$.
\begin{enumerate}
    \item Consider $\chi\geq 0$. $\frac{x^\chi}{e^{G_-}}Q(1/x,0)$ reads,
    \begin{align}
        \frac{x^\chi}{t^ie^{G_-}}Q(1/x,0)=x^{\chi}\left(1+\frac{a_1(t)}{x}+\frac{a_2(t)}{x^2}+\dots\right)\left(Q(0,0)+\frac{Q_{1,0}}{x}+\frac{Q_{2,0}}{x^2}+\dots\right)\label{chi}
    \end{align}
     $a_i(t)$ are the coefficients of series expansion of $\frac{1}{e^{G_-}}$. $[x^\geq]\frac{x^\chi}{e^{G_-}}Q(1/x,0)$ contains $\chi+1$ unknown terms such as $Q(0,0),Q_{1,0},\dots$. $[x^\geq]$ of \eqref{general 2} reads,
    \begin{align}
     t^le^{G_+}e^{G_0}Q(x,0)+b_{0}(x,t)Q(0,0)+b_{1}(x,t)Q_{1,0}+\dots b_{\chi}(x,t)Q_{\chi,0}=[x^\geq]g(x,t),
    \end{align}
    where $b_{i}$ can be calculated explicitly by \eqref{chi}. We have $\chi+1$ linearly independent solutions. We need some other relations to characterize these $\chi+1$ unknown functions $Q_{i,0}$.
    \item If $\chi < 0$, the $[x^{-n}]$ of left hand-side of \eqref{general} vanishes for $0<n\leq -\chi-1$. Thus, extra conditions
    \begin{align}
        [x^{-n}]g(x,t)=0,\label{extra}
    \end{align}
    for $n=1,2,\dots -\chi-1$ shall be satisfied. 
\end{enumerate}
Comparing \eqref{extra} with \eqref{RBVP requirement} in \cref{motivation}, the index defined by canonical factorization plays exactly the same role as the index in RBVP in \eqref{index rbvp}. It determines the number of independent solutions. However, $\chi$ is currently not well defined, as it is not unique. We solve this problem in the next section.

\subsection{Integral representation for the index}\label{sec Integral representation}
We first give an analytic insight for the $[x^>]$ operator. We denote the contour integral $\oint= \frac{1}{2\pi i}\int$ for simplicity.
\begin{lemma}
Consider $f(z)$ analytic in some annulus. Then $[z^\geq]f(z)=\oint \frac{ f(s)}{(s-z)}ds$. The contour is around $0$, inside the analytic annulus and $z$ is inside the domain surrounded by this contour. \label{lemma integral representation}
\end{lemma}
\begin{proof}
Via complex analysis,
\begin{align}
\begin{split}
    &[z^\geq]f(z)=\frac{1}{2\pi i}\int\frac{f(s)}{s}ds+\frac{z}{2\pi i}\int\frac{f(s)}{s^2}ds+\dots\\
    &=\oint\frac{f(s)}{s}\frac{1}{1-z/s}ds\\
    &=\oint\frac{f(s)}{s-z}ds
    \end{split}
\end{align}
\end{proof}

The integral representation of $[z^<]$, $[z^>]$ and $[z^0]$ are straight forward.
\begin{align}
    \begin{split}
        &[z^<]f(z)=\oint\frac{f(s)}{z-s}ds\\
        &[z^>]f(z)=\oint\frac{z}{s}\frac{f(s)}{s-z}ds\\
        &[z^0]f(z)=\oint\frac{f(s)}{s}ds
    \end{split}
\end{align}
Notice that for $[z^<]f(z)$, $z$ is outside the domain surrounded by the contour, since $|s/z|<1$.
\begin{figure}
    \centering
\begin{tikzpicture}
\tikzset{12/.style={circle, line width=1.5pt, draw=black, fill=red, inner sep=0.5pt}}
    \draw[help lines,step = 0.5] (-4.5,-2.5) grid (4.5,2.5); 
    \draw[-latex] (-5,0) -- (5,0);
    \draw[-latex] (0,-3) -- (0,3);
    \draw[red] (0,0) circle (1);
    \draw[red] (0,0) circle (2);
    \draw  (1,0) node [12]{};
    \node[above] at (1,0) {$\mathbf{C_2}$};
        \draw  (2,0) node [12]{};
    \node[above] at (2,0) {$\mathbf{C_1}$};
            \draw  (1.5,0.5) node [12]{};
    \node[above] at (1.5,0.5) {$x$};
\end{tikzpicture}
\caption{The contour is chosen as $C_1$ and $C_2$. $F$ is analytic between $C_1,C_2$. }\label{pic 11} 
\end{figure}

We denote the boundary of the annulus as contours $C_1, C_2$ and $C$ is some contour in between (See Fig \ref{pic 11}). The integral definition of index is,
\begin{align}
    -\chi_C(G(x,t))=\oint_C d\log G(x,t)=\frac{Arg(G(x,t))}{2\pi i}.\label{index combinatorics}
\end{align}
 The calculation of the argument depends on the singularities of $\log G(x,t)$. For $G_+,G_-$ to be formal series of $t$, we put the singularities whose Laurent expansions (around $t=0$) belong to $t\mathbb{C}[[t]]$ inside the contour $C_2$ and those singularities whose Laurent expansions contain $1/t$ terms outside the contour $C_1$. For those singularities whose Laurent expansions have valuation $0$ in $t$, we can put them either inside $C_2$ or outside $C_1$. Since the singularities are logarithmic singularities, we need to ensure we are integrating single valued functions by multiplying $x^{\chi_C}$.

Then by \cref{lemma integral representation}, we have an integral representation for canonical factorization,
\begin{align}
G_+&=[x^>]\log((p+q\sqrt\Delta)x^{\chi_C}/t^l)=\oint_{C_1}\frac{x}{s}\frac{\log((p+q\sqrt
{\Delta})x^{\chi_C}/t^l)}{s-x}ds\\
G_-&=[x^<]\log((p+q\sqrt\Delta)x^{\chi_C}/t^l)=\oint_{C_2}\frac{\log((p+q\sqrt
{\Delta})x^{\chi_C}/t^l)}{x-s}ds\\
G_0&=[x^0]\log((p+q\sqrt\Delta)x^{\chi_C}/t^l)=\oint_{C_1}\frac{\log((p+q\sqrt
{\Delta})x^{\chi_C}/t^l)}{s}ds
\end{align}
and
\begin{align}
\begin{split}
    &G(x,t)=G_++G_-+G_0+l\log t-\chi_C \log x\\
    &=\oint_{C_1\cup(-C_2)}\frac{\log\Big((p+q\sqrt{\Delta})x^{\chi_C}/t^i\Big)}{s-x}+l\log t-\chi_C \log x=\log(p+q\sqrt{\Delta})
\end{split}
\end{align}

We may change $\chi$ in \cref{chi section} to $\chi_C$ and the results remain the same.

\section{An Example with Non-trivial Index}\label{5 and 15}
To demonstrate how the index affects the solution of lattice walk problem, we consider the following pair of models, walk with allowed steps $\{\nearrow,\nwarrow,\downarrow\}$ and walks with allowed steps $\{\swarrow,\searrow,\uparrow\}$, both with interactions $a\neq b\neq c>1$. These two models are reverse to each other and we solved them in \cite{xu2022interacting}. There are $7$ pairs of interacting models which share similar properties: their allowed steps are reverse to each other and $\swarrow$ only appears in one model. In \cite{xu2022interacting,beaton2019quarter}, we stated that for each pair, we cannot solve the model without $\swarrow$ directly. The complete solution shall be solved by the symmetry of these two models. In this section, we show that the statement is wrong by analyzing the index.

\subsection{Walk with allow steps $\{\nearrow,\nwarrow,\downarrow\}$}
Let us first consider $\{\nearrow,\nwarrow,\downarrow\}$. The functional equation reads,
\begin{multline}
K(x,y)Q(x,y)=\frac{1}{a b c
}\Big(b c  \left(-\frac{a t}{y}+a-1\right)Q(x,0)+a c  \left(-\frac{b t y}{x}+b-1\right)Q(0,y)\\
+ (b c-a ((b-1) c+b))Q(0,0)+a b\Big)\label{model 5 func}.
\end{multline}
where
\begin{align}
K(x,y)=1-t \left(x y+\frac{y}{x}+\frac{1}{y}\right).
\end{align}
The roots are
\begin{align}
\begin{split}
&Y_0=\frac{x-\sqrt{-4 t^2 x^3-4 t^2 x+x^2}}{2 \left(t x^2+t\right)}=t+\frac{t^3 \left(x^2+1\right)}{x}+\bigO(t^5)\\
&Y_1=\frac{x+\sqrt{-4 t^2 x^3-4 t^2 x+x^2}}{2 \left(t x^2+t\right)}=\frac{x}{t \left(x^2+1\right)}-t+\bigO(t^3).
\end{split}
\end{align}
The symmetry group is 
\begin{equation}
\Big\{(x,y),\left(\frac{1}{x},y\right),\left(x,\frac{1}{(x+\frac1x)y}\right),\left(\frac1x,\frac{1}{(x+\frac1x)y}\right)\Big\}.
\end{equation}
We apply the group elements $(x,y)$ and $(1/x,y)$ to the functional equation and let $y=Y_0$. By simple calculation as we did in \cref{kernel method}

\begin{multline}\label{model 5 eq1}
0=\frac{1}{x}Q\left(\frac{1}{x},0\right)-\frac{a b t (x^2-1) Y_0^2}{cx  \left(-a t+a Y_0-Y_0\right) \left(b t Y_0-b x+x\right)}\\
+\frac{t (x^2-1) Y_0^2  (a b c+a b-a c-b c)Q(0,0)}{cx\left(-a t+a Y_0-Y_0\right) \left(b t Y_0-b x+x\right)}
-\frac{ \left(b t x Y_0-b+1\right)Q(x,0)}{b t Y_0-b x+x}.
\end{multline}
Substitute the value of $Y_0$ and by simplifications, \eqref{model 5 eq1} can be written as
\begin{align}
P_c+P_{00}Q(0,0)+P_{x,0}Q(x,0)+Q\left(\frac{1}{x},0\right)=0,\label{model 5 eq ex1}
\end{align}
where
\begin{align}
\begin{split}
P_{x,0}&=-\frac{(2 b^2 t^2 x+b^2 x^2+b^2-3 b x^2-3 b+2 x^2+2)}{2 \left(b^2 t^2+b^2 x^3-2 b x^3-b x+x^3+x\right)}-\frac{(b-1) b (x^2-1) \sqrt{- \left(4 t^2 x^2+4 t^2-x\right)/x}}{2\left(b^2 t^2+b^2 x^3-2 b x^3-b x+x^3+x\right)}.
\end{split}
\end{align}
$P_c,P_{0,0}$ are also of the form $p+q\sqrt\Delta$. In \cite{xu2022interacting}, we consider
\begin{align}
L=\log\left(-(2 b^2 t^2 x+b^2 x^2+b^2-3 b x^2-3 b+2 x^2+2)-(b-1) b (x^2-1) \sqrt{- \left(4 t^2 x^2+4 t^2-x\right)/x}\right).\label{error}
\end{align}
And write,
\begin{align}\label{model 5 eq2}
P_{x,0}=-\frac{ e^{L_+}e^{L_0}e^{L_-}}{2(b^2 t^2 + x - b x + x^3 - 2 b x^3 + b^2 x^3)}.
\end{align}
The idea of taking out the polynomial factor is a simple imitation of the kernel method for one dimensional problems \cite{prodinger2004kernel} since we understand polynomial factors very well. The calculation was performed by first taking the series expansion of $L$ around $t=0$ and then taking the series expansion of $L$ around $x=0$ by mathematica command.

The series expansions of the roots of the polynomial factors in the denominator of $P_{x,0}$ reads,
\begin{align}
\begin{split}
X_1&=\frac{1}{\sqrt{b-1}}+\frac{b^2 t^2}{2-2 b}-\frac{3 b^4 t^4}{8 (b-1)^{3/2}}-\frac{b^6 t^6}{2 (b-1)^2}+\bigO(t^8)\\
X_2&=-\frac{1}{\sqrt{b-1}}+\frac{b^2 t^2}{2-2 b}+\frac{3 b^4 t^4}{8 (b-1)^{3/2}}-\frac{b^6 t^6}{2 (b-1)^2}+\bigO(t^{8})\\
X_3&=\frac{b^2 t^2}{b-1}+\frac{b^6 t^6}{(b-1)^2}+\bigO(t^{7}).\label{roots}
\end{split}
\end{align}
Notice that both series expansions of $X_1,X_2$ contain constant terms. This implies
\begin{align}
\frac{1}{x-X_{1,2}}=\frac{-1}{X_{1,2}(1-\frac{x}{X_{1,2}})}=\frac{-1}{X_{1,2}}\sum_n\left(\frac{x}{X_{1,2}}\right)^n\label{X1}
\end{align}
is still a formal series of $t$ in a suitable annulus. So $\frac{1}{x-X_2}$ and $\frac{1}{x-X_1}$ can be expanded as formal series of $t$ with coefficients in $\mathbb{R}(x)$. $X_3$ does not contain constant term of $t$, so $1/(x-x_3)$ has to be expanded as follow,
\begin{align}
\frac{1}{x-X_3}=\frac{1}{x(1-X_3/x)}=\frac{1}{x}\sum_n\left(\frac{X_3}{x}\right)^n.
\end{align}  
Thus, if we multiply \eqref{model 5 eq ex1} by $\frac{(x-X_3)}{e^{L_-}}$, each term of the equation can be separated into $[x^>]$, $[x^<]$ and $[x^0]$ parts.
We have
\begin{align}
-\frac{e^{L_+} e^{L_0} Q(x,0)}{2(1  + b)^2(x-X_1)(x-X_2)}+[x^>]\frac{(x-X_3)}{e^{L_-}}P_{00}Q(0,0)+[x^>]\frac{(x-X_3)}{e^{L_-}}P_{c} &=0 \label{model 5 eq 3} \\
\frac{\left(x-X_3\right) Q\left(\frac{1}{x},0\right)}{e^{L_-} x}+[x^<]\frac{(x-X_3)}{e^{L_-}}P_{00}Q(0,0)+[x^<]\frac{(x-X_3)}{e^{L_-}}P_{c} &=0 \label{model 5 eq3 2}\\
Q(0, 0)-Q(0, 0) &= 0 \label{model 5 eq 3 2}.
\end{align}

 Unfortunately, this naive attempt did not give us a solution, since the $[x^0]$ part is a trivial $0=0$ equality and we cannot solve $Q(0,0)$ from the $[x^>]$, $[x^<]$ or any $[x^i]$ part by the idea of the 1-D kernel method either.

We solve $Q(0,0)$ of the reverse model instead, since $Q(0,0)$ of a walk and its reverse walk are the same (see \cref{fig Kreweras and Reverse Kreweras}). This can be proved by reversing each step of a configuration starting from $(0,0)$ and ending at $(0,0)$.
\begin{figure}
\centering
\resizebox{0.8\textwidth}{!}{
\begin{minipage}[t]{0.48\textwidth}
\centering
\begin{tikzpicture}
\tikzset{ac/.style={circle, line width=1.5pt, draw=black, fill=red, inner sep=2.5pt}}
\tikzset{bc/.style={circle, line width=1.5pt, draw=black, fill=blue, inner sep=2.5pt}}
\tikzset{cc/.style={circle, line width=1.5pt, draw=black, fill=black, inner sep=2.5pt}}
\tikzset{vert/.style={circle, line width=1.5pt, draw=black, fill=white, inner sep=2.5pt}}
\draw [gray, line width=10pt] (-0.2,5.8) -- (-0.2,-0.2) -- (5.8,-0.2);
\begin{scope}[line width=3pt, decoration={markings,mark=at position 0.65 with {\arrow{>}}}]
\draw [postaction=decorate] (0,0) node [cc] {} -- (1,1);
\draw [postaction=decorate] (1,1) node [vert] {} -- (1,0);
\draw [postaction=decorate] (1,0) node [ac] {} -- (2,1);
\draw [postaction=decorate] (2,1) node [vert] {} -- (3,2);
\draw [postaction=decorate] (3,2) node [vert] {} -- (2,3);
\draw [postaction=decorate] (2,3) node [vert] {} -- (2,2);
\draw [postaction=decorate] (2,2) node [vert] {} -- (1,3);
\draw [postaction=decorate] (1,3) node [vert] {} -- (1,2);
\draw [postaction=decorate] (1,2) node [vert] {} -- (0,3);
\draw [postaction=decorate] (0,3) node [bc] {} -- (0,2);
\draw [postaction=decorate] (0,2) node [bc] {} -- (0,1);
\draw [postaction=decorate] (0,1) node [bc] {} -- (0,0);
\end{scope}
\draw [line width=2pt, ->] (5,3) -- (6,4);
\draw [line width=2pt, ->] (5,3) -- (4,4);
\draw [line width=2pt, ->] (5,3) -- (5,2);
\end{tikzpicture}
\label{fig:paths}
\end{minipage}
\begin{minipage}[t]{0.48\textwidth}
\centering
\begin{tikzpicture}
\tikzset{ac/.style={circle, line width=1.5pt, draw=black, fill=red, inner sep=2.5pt}}
\tikzset{bc/.style={circle, line width=1.5pt, draw=black, fill=blue, inner sep=2.5pt}}
\tikzset{cc/.style={circle, line width=1.5pt, draw=black, fill=black, inner sep=2.5pt}}
\tikzset{vert/.style={circle, line width=1.5pt, draw=black, fill=white, inner sep=2.5pt}}
\draw [gray, line width=10pt] (-0.2,5.8) -- (-0.2,-0.2) -- (5.8,-0.2);
\begin{scope}[line width=3pt, decoration={markings,mark=at position 0.65 with {\arrow{<}}}]
\draw [postaction=decorate] (0,0) node [cc] {} -- (1,1);
\draw [postaction=decorate] (1,1) node [vert] {} -- (1,0);
\draw [postaction=decorate] (1,0) node [ac] {} -- (2,1);
\draw [postaction=decorate] (2,1) node [vert] {} -- (3,2);
\draw [postaction=decorate] (3,2) node [vert] {} -- (2,3);
\draw [postaction=decorate] (2,3) node [vert] {} -- (2,2);
\draw [postaction=decorate] (2,2) node [vert] {} -- (1,3);
\draw [postaction=decorate] (1,3) node [vert] {} -- (1,2);
\draw [postaction=decorate] (1,2) node [vert] {} -- (0,3);
\draw [postaction=decorate] (0,3) node [bc] {} -- (0,2);
\draw [postaction=decorate] (0,2) node [bc] {} -- (0,1);
\draw [postaction=decorate] (0,1) node [bc] {} -- (0,0);
\end{scope}
\draw [line width=2pt, ->] (5,3) -- (5,4);
\draw [line width=2pt, ->] (5,3) -- (4,2);
\draw [line width=2pt, ->] (5,3) -- (6,2);
\end{tikzpicture}
\end{minipage}
}
\caption{Examples of a configuration of $\{\nwarrow,\nearrow,\downarrow\}$ and its reverse walk $\swarrow,\searrow,\uparrow$ ending at the origin.}\label{fig Kreweras and Reverse Kreweras}
\end{figure}
\subsection{Walk with allow steps $\{\searrow,\swarrow,\uparrow\}$}
The allowed steps of the reverse model are $\{\swarrow, \searrow,\uparrow\}$. The functional equation reads
\begin{multline}
K(x,y)Q(x,y)=\frac{1}{a b c}\left(+b c \left(-at \left(\frac{x}{y}+\frac{1}{x y}\right)+a-1\right)Q(x,0)\right.\\
\left.+a c\left(-\frac{b t}{x y}+b-1\right)Q(0,y)+ \left(b c-a ((b-1) c+b)\frac{a b c t }{x y}\right)Q(0,0)+a b\right),
\end{multline}
where
\begin{align}
K(x,y)=1 - t \left(\frac{1}{xy} + \frac{x}{y} + y\right).
\end{align}
The formal series root of $K(x,y)$ is
\begin{align}
Y_0 &= -\frac{\sqrt{-4 t^2 x^3-4 t^2 x+x^2}-x}{2 t x}= t \left(x+\frac{1}{x}\right)+t^3 \left(x^2+\frac{1}{x^2}+2\right)+\bigO(t^5).
\end{align}
The symmetry group is the same as per,
\begin{align}
\Bigg\{(x,y),\left(\frac{1}{x},y\right),\left(x,\frac{(x+\frac{1}{x})}{y}\right),\left(\frac{1}{x},\frac{(x+\frac{1}{x})}{y}\right)\Bigg\}.
\end{align}
By similar calculations, we have,
\begin{align}
\begin{split}
&\frac{Q(x,0) \left(b t x-b Y_0+Y_0\right)}{b t-b x Y_0+x Y_0}-\frac{1}{x}Q\left(\frac{1}{x},0\right)=\\
&-\frac{a b t (x^2-1) Y_0}{c \left(a t x^2+a t-a x Y_0+x Y_0\right) \left(-b t+b x Y_0-x Y_0\right)}+\frac{ (a-c)b t (x^2-1) Y_0 Q(0,0)}{c \left(a t x^2+a t-a x Y_0+x Y_0\right) \left(-b t+b x Y_0-x Y_0\right)}.
\end{split}
\end{align}
Substitute the value of $Y_0$ in, we have,
\begin{align}\label{model 15 eq1}
P'_c+P'_{00}Q(0,0)+P'_{x,0}Q(x,0)-\frac{1}{x}Q\left(\frac{1}{x},0\right)=0,
\end{align}
where
\begin{align}
\begin{split}
P'_{x,0}&=\frac{(2 b^2 t^2 x+b^2 x^2+b^2-3 b x^2-3 b+2 x^2+2)}{2  \left(b^2 t^2+b^2 x^3-2 b x^3-b x+x^3+x\right)}-\frac{b(b-1) (x^2-1) \sqrt{- \left(4 t^2 x^2+4 t^2-x\right)/x}}{2\left(b^2 t^2+b^2 x^3-2 b x^3-b x+x^3+x\right)}
\end{split}
\end{align}
and $P'_{c},P_{00}$ are in the form $p+q\sqrt\Delta$. In \cite{xu2022interacting}, we consider the canonical factorization of,
\begin{equation}
L_2=\log\left((2 b^2 t^2 x+b^2 x^2+b^2-3 b x^2-3 b+2 x^2+2)-\frac{b(b-1) (x^2-1) \sqrt{- \left(4 t^2 x^2+4 t^2-x\right)/x}}{ x }\right).
\end{equation}
Then
\begin{align}
P'_{x,0}=-\frac{e^{L_{2+}}e^{L_{20}}e^{L_{2-}}}{2(b^2 t^2 + x - b x + x^3 - 2 b x^3 + b^2 x^3)}.
\end{align}
$P'_{x,0}$ has the same denominator as $P_{x,0}$. Multiply both sides of \eqref{model 15 eq1} by $\frac{x-X_3}{e^{L_{2-}}}$ and separate the $[x^>]$, $[x^<]$ and $[x^0]$ terms. This time, the $[x^0]$ term is no longer $0=0$ but a non-trivial equation. We have
\begin{align}
T_c+T_{1}Q(0,0)+T_{2}Q(0,0)+T_{3}Q(0,0)=0,
\end{align}
where
\begin{align}
T_c &= [x^0]\frac{x-X_3}{e^{L_{2-}}}P'_c\\
T_1 &= [x^0]\frac{x-X_3}{e^{L_{2-}}}P'_{x,0}\\
T_{2} &= -[x^0]\frac{x-X_3}{xe^{L_{2-}}}\\
T_3 &= [x^0]\frac{x-X_3}{e^{L_{2-}}}P'_{00}.
\end{align}
Then $Q(0,0)=-\frac{T_c}{T_{1}+T_2+T_3}$.

\subsection{Index calculation}
Although it is still possible to solve these two models by the reverse symmetry, let us consider the index of $xP_{x,0}$ and $xP'_{x,0}$ (multiply $x$ to match our theory in \cref{chi section}).
\begin{align}
    \begin{split}
        xP_{x,0}&=-\frac{x(2 b^2 t^2 x+b^2 x^2+b^2-3 b x^2-3 b+2 x^2+2)}{2 \left(b^2 t^2+b^2 x^3-2 b x^3-b x+x^3+x\right)}-\frac{(b-1) b (x^2-1) \sqrt{-x \left(4 t^2 x^2+4 t^2-x\right)}}{2 \left(b^2 t^2+b^2 x^3-2 b x^3-b x+x^3+x\right)}\\
        xP'_{x,0}&=\frac{x(2 b^2 t^2 x+b^2 x^2+b^2-3 b x^2-3 b+2 x^2+2)}{2  \left(b^2 t^2+b^2 x^3-2 b x^3-b x+x^3+x\right)}-\frac{b(b-1) (x^2-1) \sqrt{-x \left(4 t^2 x^2+4 t^2-x\right)}}{2\left(b^2 t^2+b^2 x^3-2 b x^3-b x+x^3+x\right)}.
    \end{split}\label{index choose circle}
\end{align}
Notice that $xP_{x,0}$ and $xP'_{x,0}$ are conjugate (under Galois transform) to each other. The product of the numerator of $P_{x,0}$ and $P'_{x,0}$ is 
\begin{align}
    -4x \left(b^2 t^2+b^2 x^3-2 b x^3-b x+x^3+x\right) \left(b^2 t^2 x^3+b^2-b x^2-2 b+x^2+1\right).
\end{align}
The series expansion of roots of $\left(b^2 t^2 x^3+b^2-b x^2-2 b+x^2+1\right)=0$ are,
\begin{align}
\begin{split}
    & J_1=\frac{b-1}{b^2 t^2}-b^2 t^2-\frac{2 b^6 t^6}{b-1}+O\left(t^7\right)\\
    &J_2=-\sqrt{b-1}+\frac{b^2 t^2}{2}-\frac{5 b^4 t^4}{8 \sqrt{b-1}}+\frac{b^6 t^6}{b-1}+O\left(t^7\right)\\
    &J_3=\sqrt{b-1}+\frac{b^2 t^2}{2}+\frac{5 b^4 t^4}{8 \sqrt{b-1}}+\frac{b^6 t^6}{b-1}+O\left(t^7\right).
\end{split}
\end{align}
The series expansion of roots of $ \left(b^2 t^2+b^2 x^3-2 b x^3-b x+x^3+x\right)=0$ are $X_1,X_,X_3$ from \eqref{roots}. The zeros of $P_{x,0},P'_{x,0}$ come from $0, X_1,X_2,X_3,J_1,J_2,J_3$ and the singularities come from $X_1,X_2,X_3$.

We denote the roots of $\sqrt{-x \left(4 t^2 x^2+4 t^2-x\right)}$ as $0,x_1,x_2$. Direct calculation shows $x_1=4t^2+O(t^6)$ and $x_2=\frac{1}{4t^2}+O(t^2)$. $P_{x,0}$ and $P'_{x,0}$ have branch cuts $[0x_1]$ and $[x_2\infty]$ and we shall include $[0x_1]$ inside the contour $C$. Recall $X_3=\frac{b^2 t^2}{b-1}+\bigO(t^{6})$ in \eqref{roots}. For $b>1$, the leading term of $X_3$ is always greater than the leading term of $x_1$ unless $b=2$. For $b=2$, $X_3-x_1\sim |O(t^{10})|$. Thus, for $t$ small enough, $X_3$ do not involve the branch cut. $J_1$ is always outside the contour $C$ and does not contribute. We ignore $J_1$ in the following calculation. $0$ is at the end of the branch cut, so we shall consider it carefully. Let us denote the numerator of $xP_{x,0}$ and $xP'_{x,0}$ as
\begin{align}
    \begin{split}
        &Nu(P_{x,0})=x p(x,b,t)+q(x,b,t)\sqrt{x(x-x_1)(x_2-x)}\\
        &Nu(P'_{x,0})=x p(x,b,t)-q(x,b,t)\sqrt{x(x-x_1)(x_2-x)}
    \end{split}
\end{align}
 For convenience, we write,
\begin{align}
    \begin{split}
        &Nu(xP_{x,0})=p(x,b,t)\sqrt{x}\left(\sqrt{x} +\frac{q(x,b,t)}{p(x,b,t)}\sqrt{(x-x_1)(x_2-x)}\right)\\
        &Nu(xP'_{x,0})=p(x,b,t)\sqrt{x}\left(\sqrt{x}-\frac{q(x,b,t)}{p(x,b,t)}\sqrt{(x-x_1)(x_2-x)}\right)\label{argument}
    \end{split}
\end{align}
$q(x,b,t)<0$ for $|x|<1$. The sign of $q(x,b,t)$ is fixed. $p(x,b,t)$ reads,
\begin{align}
    p(x,b,t)=2 b^2 t^2 x+(b-2) (b-1) x^2+(b-2) (b-1).
\end{align}
The determinant (as a function of $x$) of $p(x,b,t)=4 b^4 t^4-4 \left(b^2-3 b+2\right)^2<0$ if $b\neq 0$ and $t$ is some small positive value. Thus, the sign of $p(x,b,t)$ is characterized by $b$. If $b=2$, $p(x,b,t)=8t^2x\geq 0$ on $[0x_1]$. $P_{x,0}$ and $P'_{x,0}$ shall be factored as 
\[8t^2\sqrt{x}(x\sqrt{x}\pm \frac{q(x,b,t)}{8t^2}\sqrt{(x-x_1)(x_2-x)}).\]
The factor $p(x,b,t)$ (or $8t^2$) in the front does not contribute to the variation of argument and we assume it is $1$.  

Let $x$ go around the branchcut $[0x_1]$ anticlockwise starting from $0$ with argument $0$. Around $0$, the variation of the argument is determined by $\sqrt{x}$ since $\pm\frac{q(x,b,t)}{p(x,b,t)}\sqrt{(x-x_1)(x_2-x)}$ inside the bracket do not have argument change. So, if $x$ rotates around $x=0$ anticlockwise by $2\pi$, both $P'_{x,0}$ and $P_{x,0}$ rotate anticlockwise by $\pi$.

Then $x$ go along the branch cut $\lfloor0x_1\rfloor$ from below anticlockwise. $\sqrt{x}=\sqrt{e^{2\pi i}|x|}<0$ and it is real and does not have the variation of the argument along $\lfloor0x_1\rfloor$. $\sqrt{(x-x_1)(x_2-x)}$ shall be considered as $-i\sqrt{|(x-x_1)(x-x_2)|}$. The rotation of $Nu(xP_{x,0})$ and $Nu(xP'_{x,0})$ both end on the negative real axis but start with the opposite point on the imaginary axis. We have the following result. For $x\in\lfloor0x_1\rfloor$,
\begin{enumerate}
    \item If $1<b<2$, $p(x,b,t)<0$. $-\frac{q(x,b,t)}{p(x,b,t)}\sqrt{|(x-x_1)(x_2-x)|}<0$. Then $Nu(xP_{x,0})$ starts from the negative imaginary axis and ends on the negative real axis. The variation of the argument is $\frac{\pi}{2}$ clockwise. $Nu(xP'_{x,0})$ starts from the positive imaginary axis and ends on the negative real axis. The variation of the argument is $\frac{\pi}{2}$ anticlockwise.
    \item If $b\geq 2$, $p(x,b,t)>0$. $-\frac{q(x,b,t)}{p(x,b,t)}\sqrt{|(x-x_1)(x_2-x)|}>0$. Then $Nu(xP_{x,0})$ starts from the positive imaginary axis and ends on the negative real axis. The variation of argument is $\frac{\pi}{2}$ anticlockwise. $Nu(xP'_{x,0})$ starts from the negative imaginary axis and ends on the negative real axis. The variation of argument is $\frac{\pi}{2}$ clockwise.
\end{enumerate}

Then, consider the variation of the argument near $x=x_1$, the second term in the bracket of \eqref{argument} is $0$ and the variation of argument depends on $\sqrt{x}$. $\sqrt{x}$ do not have variation of argument at $x=x_1$.

Above the branch cut $\lceil 0x_1\rceil$, $\sqrt{x}=\sqrt{e^{2\pi i}|x|}>0$ and $\sqrt{(x-x_1)(x_2-x)}$ are still considered as $-i\sqrt{|(x-x_1)(x-x_2)|}$ since the relative position of $x$ to $x_1,x_2$ do not change. We have results similar to $\lfloor0x_1\rfloor$ case.

Thus, for the whole contour, we may conclude,
\begin{enumerate}
    \item If $1<b<2$, $Nu(xP'_{x,0})$ rotate $2\pi$ anticlockwise.  $Nu(xP_{x,0})$ does not have the variation of the argument. $0$ contributes to the index of $xP'_{x,0}$.
    \item If $b\geq 2$, $Nu(xP'_{x,0})$ does not have the variation of the argument.  $Nu(xP_{x,0})$ rotate $2\pi$ anticlockwise. $0$ contributes to the index of $xP_{x,0}$.
\end{enumerate}

Other roots are poles. we substitute them into $P_{x,0}$ and $P'_{x,0}$ to see whether they are really roots or not. We denote the roots $a_i$ and singularities $b_i$ of a function as $f=\frac{\{a_i\}}{\{b_i\}}$. Then, by direct calculation, we have,
\begin{enumerate}
    \item If $1<b<2$,
    \begin{align}
    \begin{split}
        &xP'_{x,0}=\frac{\{0,J_2,J_3,X_3\}}{\{X_1,X_2,X_3\}}\\
        &xP_{x,0}=\frac{\{X_1,X_2\}}{\{X_1,X_2,X_3\}}
    \end{split}
    \end{align}
     \item If $b\geq 2$,
    \begin{align}
    \begin{split}
        &xP'_{x,0}=\frac{\{J_2,J_3\}}{\{X_1,X_2,X_3\}}\\
        &xP_{x,0}=\frac{\{0,X_1,X_2,X_3\}}{\{X_1,X_2,X_3\}}
    \end{split}
    \end{align}
\end{enumerate}
In addition, $|J_2|,|J_3|$ are smaller than $|X_1|, |X_2|$ when $1<b<2$ by comparing the leading term of the $t$ series expansions. Likewise, $|J_2|,|J_3|$ are greater than $|X_1|, |X_2|$ when $b>2$.

Thus, the exact solution space is,
\begin{enumerate}
    \item when $1<b<2$, $\chi_Cx(P_{x,0})=1$ for any contour $C$ and $\chi_C(xP'_{x,0})=-1,-2,-3$, depending on whether we put $J_2,J_3,X_1,X_2$ inside the contour. We can solve $Q(0,0)$ by the functional equation of $P'$ but not $P$.
    \item when $b\geq 2$, $\chi_C(xP'_{x,0})=1,2,3$ depending on $C$ and $\chi_C(xP_{x,0})=-1$ since poles and zeros $X_1,X_2,X_3$ coincide. We can solve $Q(0,0)$ by the functional equation of $P$ but not $P'$.
\end{enumerate}
In the calculation \eqref{error}, we factorize $L$ by simply plunging them into Mathematica and force $L=L_++L_0+L_-$. Now we find that the error is that the roots of the denominator coincide with the roots of $L$ and the factorization $L=L_+L_0L_-$ is not well defined with the series expansion of $X_1,X_2,X_3$.

\subsection{A kernel method criteria for linear independent solutions}\label{kernel method criteria}
Recall the original kernel method \cite{prodinger2004kernel} by a simple example. Suppose we have a functional equation
\begin{equation}
(tx^2-x+t)F(x,t)=tF(0,t)-x.\label{kernel method original:functional}
\end{equation}
$F(x,t)=\sum_{i,n}f_{i,n}x^it^n$ is defined as a formal series in $t$ with polynomial coefficients in $x$. To solve $F(x,t)$, we consider the roots of $tx^2-x+t=t(x-x_0)(x-x_1)$
One of the roots have the following series expansions around $t=0$,
\begin{equation}\label{x01}
x_0 = t+t^3+2 t^5+\bigO(t^6)
\end{equation}
We substitute $x_0$ into \eqref{kernel method original:functional}. This gives
\begin{equation}
tF(0,t)-x_1=0.
\label{kernel method original}
\end{equation}

Now assume that we have a ghost $F(1/x,0)$ in \eqref{kernel method original:functional}. Then this is a cRBVP with $\chi_C=-1$ and \eqref{kernel method original} is the extra condition \eqref{extra} that the cRBVP satisfies.

\section{A Combinatorial Analog of Conformal Gluing Function}\label{comparing}
In previous sections, we discussed the combinatorial analog of `index' and show how they affect the solution space of a lattice walk model. Now we turn to the combinatorial analog of the conformal gluing function. We prove the following result in this section.

\begin{theorem}\label{part 2 th 1}
    Consider a quarter-plane lattice walk problem associated with a finite group. The generating functions $Q(x,0)$, $Q(0,y)$, $Q(x,y)$ are defined as formal series in $\mathbb{C}[x,y][[t]]$. There exist a conformal mapping $z=w(x)$ that maps $x$ to a formal series of $1/z$ and $Y_0(x)$ to a formal series of $z$\footnote{ Another choice is mapping $x$ to a formal series of $z$ and $Y_0(x)$ to a formal series of $1/z$.}. For models associated with $D_2$ groups, $w(x)$ has an explicit representation,
    \begin{align}
        x+1/x=z.
    \end{align}
The functional equation after the conformal mapping reads,
\begin{align}
    F(1/z)+G(z)H(z)=g(z),
\end{align}
where $F(1/z)=Q(x(z),0)$ and $H(z)=Q(0,Y_0(x(z)))$. $G(z),g(z)$ are known algebra functions in $z,t$. We can apply the canonical factorization and solve $F(1/z)$, $H(z)$.
\end{theorem}
This theorem also holds for the functional equation for $Q(X_0(y)),Q(0,y)$ with the conformal gluing function $y+1/y=z$ by symmetry.

The condition ``as a formal series of $t$" can also be switched to an analyticity condition,
\begin{theorem}\label{part 2 th 2}
    Consider a quarter-plane lattice walk problem associated with a finite group. The generating functions $Q(x,0),Q(0,y),Q(x,y)$ are defined as meromorphic functions in $x,y$ with fixed value of $t$. The conformal mapping $z=w(x)$ maps some convergent domain of $Q(x,0)$ to a cut plane where $Q(x(z),0)$ is analytic in the plane except the cut and $Q(0,Y_0(x(z)))$ is analytic in some domain around this cut. For models associated with $D_2$ group, $w(x)$ has an explicit representation,
    \begin{align}
        x+1/x=z. \qquad 
    \end{align}
The functional equation after the conformal mapping reads,
\begin{align}
    Q(x(z),0)+G(x(z),t)Q(0,Y_0(x(z)))=g(x(z),t),
\end{align}
This equation is a combinatorial analog of Riemann Boundary value problem in $z$.
\end{theorem}

\subsection{Positive term extraction and conformal mapping}\label{combinatorial cfg}
 Recall the kernel method in \cref{kernel method}. The idea of the kernel method is to separate the $[x^>]$ and $[x^<]$ terms of an equation. This requires that after some calculation, the unknown functions remaining in the equation should be $Q(x,0)$ and $Q(1/x,0)$. 
 
 However, we may not be able to find such an equation even in a lattice walk problem. For Gessel's walk \cite{bousquet2016elementary}, the equation derived from the kernel method is,
\begin{align}
    t(1+Y_1(x))Q(0,Y_1(x))-t(1+x^2Y_1(x))Q(0,x^2Y_1(x))=\frac{1}{x}-x.
\end{align}
It is a relation between $Q(0,Y_1(x))$ and $Q(0,x^2Y_1(x))$. $Y_1(x)$ is neither a formal series of $x$ nor a formal series of $1/x$. We cannot apply the operators $[x^>],[x^<]$.

The conformal mapping in RBVP provides a way to deal with this situation in combinatorics. Let us demonstrate this with a simple example.

Recall \eqref{func simple},
\begin{align}
    xY_0(x)-txQ(x,0)-tY_0Q(0,Y_0(x))=0\label{func simple again}
\end{align}
Consider the following substitution $z=w(x)$ (the inverse is denoted as $x=w^{-1}(z)=x(z)$),
\begin{align}
&z=w(x)=x+1/x\\
&Y_0(x(z))=\frac{1}{2t}-z-\frac{\sqrt{(tz-1)^2-4t^2}}{2t}.\label{simple mapping}
\end{align}
$Y_0(x(z))$ can be expanded as a formal series of $t$ with $[z^\geq]$ coefficients,
\begin{align}
    Y_0=t+zt^2+(1+z^2)t^3+O(t^4),
\end{align}
We have two choice for $x=w^{-1}(z)$. Our aim is to choose an $x(z)\in 1/z\mathbb{C}[[1/z]]$. This condition is satisfied by choosing,
\begin{align}
x(z)=\frac{1}{2}(z-\sqrt{z^2-4})=\frac{1}{z}+\frac{1}{z^3}+O\Big(\frac{1}{z}\Big)^4
\end{align}
The Laurent expansion of $x(z)$ requires $z\geq 2$. This choice of the branch of $w^{-1}(z)$ reveals a conformal mapping from the domain inside the unit circle on the $x$-plane, to the whole $z$-plane with infinity except for a cut $[-2,2]$.

One may also interpret the results in an analytic insight. Consider meromorphic functions. We choose the mapping $x=w^{-1}(z)$ such that $Y_0(x(z))$ is analytic at $0$ and $x(z)$ is analytic at $\infty$. Notice that $x(z)$ and $Y_0(x(z))$ have only square root singularities and the branch cuts are $[-2,2]$ and $[1/t-2,1/t+2]$. If $t<1/4$, these two branch cuts do not intersect. Then, $x(z)$ and $Y_0(x(z))$ can be simultaneously expanded as $1/z$ and $z$ in the annulus centered at $0$ with $[-2,2]$ in the finite exterior domain and $[1/t-2,1/t+2]$ in the infinite exterior domain. See \cref{pic 1}.
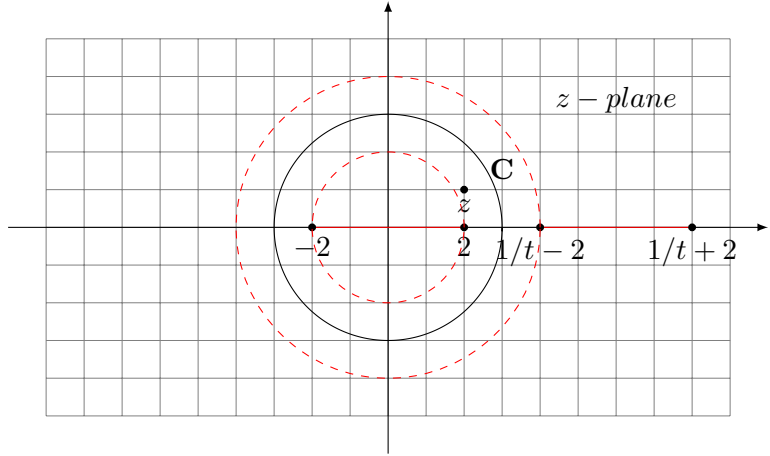
\begin{figure}
    \centering
\begin{tikzpicture}
\tikzset{12/.style={circle, line width=1.5pt, draw=black, fill=red, inner sep=0.5pt}}
    \draw[help lines,step = 0.5] (-4.5,-2.5) grid (4.5,2.5); 
    \draw[-latex] (-5,0) -- (5,0);
    \draw[-latex] (0,-3) -- (0,3);
    \draw[red](-1,0)--(1,0);
    \coordinate (a0) at (-1,0);
    \coordinate (b0) at (1,0);
    \node[below] at (a0) {$-2$};
    \node[below] at (b0) {$2$};
    \draw  (-1,0) node [12]{};
    \draw  (1,0) node [12]{};
    \draw[red](2,0)--(4,0);
    \coordinate (a1) at (2,0);
    \coordinate (b1) at (4,0);
    \node[below] at (a1) {$1/t-2$};
    \node[below] at (b1) {$1/t+2$};
    \draw  (2,0) node [12]{};
    \draw  (4,0) node [12]{};
    \draw[red][dashed] (0,0) circle (1);
    \draw[domain = -2:360][samples = 300] plot({1.5*cos(\x)}, {1.5 *sin(\x)});
    \draw[red][dashed] (0,0) circle (2);
    \coordinate (z) at (3,2);
    \node[below] at (z) {$z-plane$};
    \draw  (1,0.5) node [12]{};
    \node[below] at (1,0.5) {$z$};
    \node[above] at (1.5,0.5) {$\mathbf{C}$};
\end{tikzpicture}
\caption{We choose the contour $C$ of the integral (black circle) to be inside the annulus (between dashed circles). The Laurent expansion of $x(z)\in \mathbb{C}[[1/z]][[t]]$ and $Y_0(x(z))\in\mathbb{C}[[z]][[t]]$.}\label{pic 1} 
\end{figure}

The functional equation reads,
\begin{align}
    x(z)Y_0(x(z))-tx(z)Q(x(z),0)-tY_0(x(z))Q(0,Y_0(x(z)))=0.\label{simple z}
\end{align}
We immediately have,
\begin{align}
\begin{split}
    &tY_0(x(z))Q(0,Y_0(x(z)))=[z^>]x(z)Y_0(x(z))\\
    &tx(z)Q(x(z),0)=[z^<]x(z)Y_0(x(z))\label{model1 result}
\end{split}
\end{align}
By substitution, we solve $Q(x,0)$ and $Q(0,Y_0(x))$,
\begin{align}
\begin{split}
    &tY_0(x)Q(0,Y_0(x))=[z^>]x(z)Y_0(x(z))|_{z=x+1/x}\\
    &txQ(x,0)=[z^<]x(z)Y_0(x(z))|_{z=x+1/x}\label{z}
\end{split}
\end{align}
One can check that this representation of $Q(x,0)$ and $Q(0,Y_0(x))$ coincides with the results \eqref{simple result} from the kernel method by interpreting $[z^\geq],[z^<]$ as a contour integral,
\[tY_0(x)Q(0,Y_0(x))=\oint_C \frac{x(s)Y_0(x(s))}{s-z}ds|_{z=x+1/x}.\]
$s$ integral on a contour $C$ within the annulus with $2<|s|<1/t-2$. Inside $C$, we have,
\begin{enumerate}
    \item A pole $z=x+1/x$ (see \cref{pic 1}). This gives $xY_0(x)$.
    \item A branch cut $[-2,+2]$. This gives $\oint_{-2}^{+2} \frac{x(s)Y_0(x(s))}{s-z}ds$. $\oint_{-2}^{2}$ refers to a contour integral around the branch cut. 
\end{enumerate}
The second part can be transformed into a contour integral in the $x$-plane through conformal mapping $z=x+1/x$.
\begin{align}
    \oint_{-2}^{2} \frac{x(s)Y_0(x(s))}{z-s}ds=-\oint \frac{w^{-1}(w(u))Y_0(w^{-1}(w(u)))}{(u+1/u)-(x+1/x)}\frac{ds}{du}du=-\oint \frac{Y_0(u)(u-1/u)}{(u+1/u)-(x+1/x)}du\label{x result}
\end{align}
The contour in $x$-plane is the circle $|u|=1$. The mapping $u=\frac{1}{2}(z-\sqrt{z^2-4})$ maps the lower part of the branch cut $[-2,2]$ to the upper half circle in the$x$-plane and the upper part of the branch cut to the lower half circle in the $x$-plane.

Notice,
\begin{align}
    \frac{(u-1/u)}{(u+1/u)-(x+1/x)}=x\Big(\frac{1}{u-x}-\frac{1}{1/u-x}\Big)\label{A B}
\end{align}
We leave the first part of \eqref{A B} unchanged and consider changing the variable $v=1/u$ in the second part.
\begin{align}
    \oint \frac{xY_0(u)}{x-1/u}du=\oint \frac{xY_0(v)}{x-v}\frac{1}{v^2}dv\label{uv change}
\end{align}
Since $|u|=|v|=1$, $v=1/u$ lies on the same contour in the opposite direction. And also notice $Y_0(u)=Y_0(1/u)$. Thus, if we write the contour integral of both parts in the same direction, \eqref{x result} reads,
\begin{align}
    -\oint \frac{Y_0(u)(u-1/u)}{(u+1/u)-(x+1/x)}du=-\oint\frac{xY_0(u)}{u-x}(1-1/u^2)du=-\oint \frac{Y_0(u)}{u-x}\frac{x}{u}(u-1/u)du
\end{align}
By applying \cref{lemma integral representation} in reverse, we have
\begin{align}
-\oint \frac{Y_0(u)}{u-x}\frac{x}{u}(u-1/u)du=-[x^>]Y_0(x)(x-1/x)
\end{align}
Combining with the $xY_0(x)$ from the residue, and \eqref{simple z} we recover \eqref{solution kernel}.

\subsection{Analytic insight of combinatorial problems}
$z=x+1/x$ is a special function for the simple lattice walk. We need to generalize this. A candidate is the conformal gluing function $x(z)=w^{-1}(z)$ defined in the RBVP \eqref{conformal}.

Let us compare the effect of the conformal mapping in combinatorics \eqref{simple mapping} and the conformal gluing function in RBVP. In RBVP, the idea of the conformal gluing function is to glue two different functions onto the same line. while in the kernel method, we are trying to separate some functions. $x$ and $1/x$ cannot be separated by substitution, but $x$ and $Y_0(x)$ can. The conformal mapping is applied to \eqref{func simple again} but not to \eqref{sep}.

We may describe the idea of the conformal mapping in combinatorics as follows. Recall \eqref{func simple} again.
\begin{align}
    xY_0-txQ(x,0)-tY_0Q(0,Y_0)=0
\end{align}
We obtain \eqref{func simple} by substitute $y=Y_0(x)$ as a formal series of $t$ into the functional equation \eqref{func model1}. From analytic insight, \eqref{func simple} is well defined within some annulus bounded by $C_1\cup C_2$ with radius $|C_1|>|C_2|$. $Y_0(x)\in \mathbb{C}(x,1/x)[[t]]$ is because there is a branch cut $[x_1x_2]$ inside the domain bounded by $C_2$. If we take $[x^>]$, the contour integral $C$ lies between $C_1,C_2$ and we cannot separate $[x^>]$ and $[x^<]$ of $Q(0,Y_0(x))$. However, if we can find a mapping from the $x$-plane to $z$-plane such that the branch cut of $Y_0(x)$ does not lie inside the contour of the integral, then $[z^>]Q(,Y_0(x(z)))$ can be calculated. We are unable to remove the cut $[x_1x_2]$, so we glue $C_1$, which is $X_0([y_1y_2])$ and make $Y_0(x(z))$ analytic on this curve. This means that we are still using the curve $X_0([y_1y_2])$ and the conformal gluing function $z=w(x)$. 

We refer to \cite{raschel2012counting} for a detailed proof of the existence and uniqueness of this curve, since $Y_0(X_0(y))=y$ is not a well-defined substitution of formal series in general.

\subsection{Conformal gluing functions as combinatorial conformal mappings}
After gluing, the last step we need is to use M\"{o}bius transform to map the branch cut away from $0$ to ensure that $Y_0(x(z))$ is analytic around $0$ and everything we want is a formal series in $t$. We get the following theorem,
\begin{theorem}\label{part 2 th 0}
    Consider a quarter-plane lattice walk model. The generating functions are defined by \eqref{functional}. $Y_0(x)$($X_0(y)$) are the roots of the kernel $K(x,y)$ and are formal series in $t$. $\phi$ is the Galois automorphism of $XyK(X,y)=0$ In the $x$-plane, $X_0([y_1y_2]),X_1([y_1y_2])$ coincide and are symmetric by the real axis. There exists a conformal mapping $z=w(x)$ such that $w(s)=w(\phi(s))$ for any $s\in X_0([y_1y_2])$. $w(x)$ further has the following property,
    \begin{enumerate}
        \item $z=w(x)$ maps the domain inside the contour $X_0([y_1y_2])$ to $z$-plane except a segment $[ab]$.
        \item $Q(x(z),0)$ is analytic on $z$-plane except a segment $[ab]$.
        \item $Q(Y_0(x(z)),0)$ is analytic in some small domain around $[ab]$.
   \end{enumerate}
Or, if we consider $Q(x(z),0)$, $Q(0,Y_0(x(z)))$ as formal series of $t$, we have,
  \begin{enumerate}
        \item $z=w(x)$ maps the domain inside the contour $X_0([y_1y_2])$ to $z$-plane except a segment $[ab]$.
        \item $Q(x(z),0)$ is in $\mathbb{C}[[1/z]][[t]]$.
        \item $Q(Y_0(x(z)),0)$ is in $\mathbb{C}[[z]][[t]]$.
\end{enumerate}
\end{theorem}
\begin{proof}
    The existence and the gluing property of $w(x)$ was proved in \cite{raschel2012counting} and the authors also gave explicit expression of $w(x)$. Here we prove the second and third property which is related to the positive degree term extraction.
    
    Since the mapping $z=w(x)$ is conformal, the preimage of the whole $z$-plane except $[ab]$ is the domain inside $X_0([y_1y_2])$ and $Q(x,0)$ is convergent in this domain, $Q(x(z),0)$ converges in $z$-plane except the branch cut. 

    $Q(0,Y_0(x))$ is analytic on $[ab]$ due to the gluing property. Since $w(s)=w(\phi(s))$, $u^-=w(s)$ and $u^+=w(\phi(s))$ for $s\in X_0([y_1y_2])$. Then $Q(0,Y_0(x(u^-)))=Q(0,Y_0(x(u^+)))$ and it is analytic on $[ab]$. Since $Q(0,Y_0(x))$ only has $[x_1x_2]$ as branch cut and it is away from $X_0([y_1y_1])$, $Q(0,Y_0(x(z)))$ is analytic in some domain around $[ab]$.

The combinatorial insight of $w(x)$ can be treated as a M\"{o}bius transform of $w(x)$ in analytic insight. By \textbf{theorem 5.27} in \cite{fayolle2017random}, the conformal gluing function $w(x)$ in analytic insight can be written as,
    \begin{align}
        w(x)=\frac{1}{x-x_0}+f(x)
    \end{align}

where $f(x)$ is analytic inside $X_0([y_1y_2])$. $w(x)$ has a singularity at $x_0$, which maps to $\infty$ in $z$-plane. We consider a M\"obius transform,
\begin{align}
    w'(x)=\frac{aw(x)+b}{cw(x)+d}=\frac{\frac{a}{x-x_0}+af(x)+b}{\frac{c}{x-x_0}+cf(x)+d}=\frac{a+axf(x)+bx-af(x)x_0-bx_0}{c+cxf(x)+dx-cf(x)x_0-dx_0}
\end{align}
And choose $c-cf(0)x_0-dx_0=0$, $a-af(0)x_0-bx_0\neq 0$. Then $w'(x)$ still has the property that $w'(x)$ glues the contour $X_0(y_1y_2)$ to a segment on $z$-plane. $x=0$ is mapped to $z=w'(x)=\infty$. Now we expand $x(z)=w'^{-1}(z)$ as a formal series of $t$ in some annulus where $[ab]$ is in the bounded domain of the complement of the annulus. If
\begin{align}
    x(z)=a_0(t)+g(1/z,t)
\end{align}
where $g(1/z,t)\in \mathbb{C}[[1/z]][[t]]$. It is a formal series of $1/z$ since $w'^{-1}(z)$ is analytic at $\infty$. We further claim $a_0(t)=0$ since $w'^{-1}(\infty)=0$. 

Now let us consider $Y_0(x)$. On $x$-plane inside $X_0(y_1y_2)$, $Y_0(x)$ is singular on $[x_1x_2]$. $[x_1x_2]$ is mapped to $[z_1z_2]$ on $z$-plane. We now adjust the radius of the analytic annulus (see \cref{pic 121}) such that $[z_1z_2]$ is in the unbounded domain of the complement of the annulus. This can always be done since the pre-image of $[ab]$ and $[z_1z_2]$ do not intersect. Further we choose $x_s$ such that,
\begin{enumerate}
    \item $x_s$ is not on the segment $[x_1x_2]$.
    \item $|x_s|$ is some fixed value for $t$ small enough (For example, we may choose some point not on the real line).
\end{enumerate}
And apply translation $z\to z+z_0$ such that $0=w(x_s)$ on $z$-plane. Then in the same annulus, $Y_0(x(z))$ is expanded as a formal series of $t$,
\begin{align}
    Y_0(x(z))=Y_0(x_s)+z h(z,t).
\end{align}
$h(z)$ only contains $[z^\geq]$ terms since $Y_0(x(z))$ is analytic around $[ab]$. $Y_0(x_s)\in t\mathbb{C}[[t]]$ since $Y_0(x)\in t\mathbb{C}(x)[[t]]$. Thus, $Y_0(x(z))\in t\mathbb{C}[[t]]+z\mathbb{C}[[z]][[t]]$. The composition law of the iterated Laurent series \cite{xin2004ring} shows $Q(Y_0(x(z)),0)\in\mathbb{C}[[z]][[t]]$.
\begin{figure}
    \centering
\begin{tikzpicture}
\tikzset{12/.style={circle, line width=1.5pt, draw=black, fill=red, inner sep=0.5pt}}
    \draw[help lines,step = 0.5] (-4.5,-2.5) grid (4.5,2.5); 
    \draw[-latex] (-5,0) -- (5,0);
    \draw[-latex] (0,-3) -- (0,3);
    \draw[red](-1,0)--(1,0);
    \coordinate (a0) at (-0.5,-0.5);
    \coordinate (b0) at (0.5,0);
    \node[below] at (a0) {$a$};
    \node[below] at (b0) {$b$};
    \draw  (-0.5,-0.5) node [12]{};
    \draw  (0.5,0) node [12]{};
    \draw[red](-0.5,-0.5)--(0.5,0);
    \coordinate (a1) at (3,0);
    \coordinate (b1) at (4,1);
    \node[below] at (a1) {$z_1$};
    \node[below] at (b1) {$z_2$};
    \draw  (3,0) node [12]{};
    \draw  (4,1) node [12]{};
    \draw[red](3,0)--(4,1);
    \filldraw[red,pattern color=red, pattern=north east lines, even odd rule](0,0) circle (1)(0,0) circle (2);
    \coordinate (z) at (3,2);
    \node[below] at (z) {$z-plane$};
    \node[above] at (1.5,0.5) {$z$};
\end{tikzpicture}
\caption{In side the red annulus $Q(x(z),0)\in \mathbb{C}[[1/z]][[t]]$ and $Q(Y_0(x(z)),0)\in \mathbb{C}[[1/z]][[t]]$.}\label{pic 121} 
\end{figure}
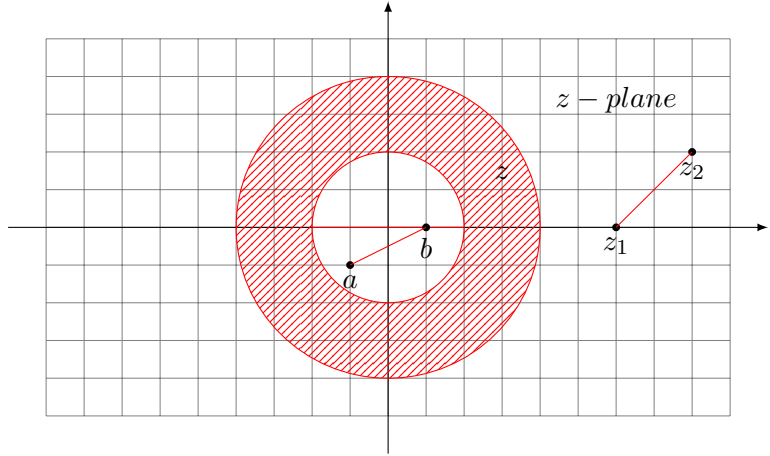
\end{proof}
\cref{part 2 th 1} and \cref{part 2 th 2} are direct corollaries of \cref{part 2 th 0}. 

\section{From Conformal Mapping to Tutte's Invariant}\label{tutte}
The algebraic property of the generating functions is widely interested in combinatorics. There are different techniques to find algebraic models through computer algebra, the kernel method and RBVP.etc.\cite{bostan2017human,bousquet2016elementary,bostan20163}. Here we show that the combinatorial analog of RBVP coincides with Tutte's invariant approach in detecting algebraic properties and the conformal mapping is the weak invariant defined in \cite{raschel2020counting}.

Recall we are solving a equation
\begin{align}
    x(z)Y_0(x(z))-tx(z)Q(x(z),0)-tY_0(x(z))Q(0,Y_0(x(z)))=0
\end{align}
after conformal mapping. We apply $[z^>]x(z)Y_0(x(z))$ to obtain the solution. From an analytic point of view, this integral representation is not algebraic since it is an elliptic integral. If $x(z)Y_0(x(z))\sim F(x(z))+G(Y_0((z))$, then the integral gives algebraic functions in $z$ (for $F(x(z))$, we can consider the integral around infinity). This is exactly the decoupling criteria of the Tutte's invariant method \cite{raschel2020counting}.

Let us take Kreweras walk as an example. The allowed steps are $\{\nearrow,\leftarrow,\downarrow\}$. The kernel reads $K(x,y)=1-t(1/x+1/y+xy)$. This shows,
\begin{align}
    x(z)Y_0(x(z))=\frac{1}{t}-\frac{1}{x(z)}-\frac{1}{Y_0(x(z))}.
\end{align}
Consider the analytic interpretation of $[z^>]$ and $[z^<]$. The integral of $1/Y_0(x(z))$ does not have branch-cuts inside the contour and the integral of $1/x(z)$ can be calculated by the residues outside the contour. Thus, $[z^>]x(z)Y_0(x(z))$ is algebraic. If $z=w(x)$ is also an algebraic function, $Q(0,y)$ is algebraic.

A more difficult example is Gessel's walk. The allowed steps are $\{ \nearrow,\leftarrow,\rightarrow,\swarrow\}$. The kernel reads $K(x,y)=1-t(x+1/x+xy+1/(xy))$. We are still considering $x(z)Y_0(x(z))$. Fortunately, in \cite{raschel2020counting}, Gessel's walk is decoupled as,
\begin{align}
    xY_0(x)=-\frac{1}{t(1+Y_0(x))}+\frac{1}{t}-\frac{1}{x}.
\end{align}
The generating functions $Q(x,0),Q(0,y)$ and $Q(x,y)$ of Gessel's walk are algebraic.

The criteria may also be combined with the discussion of the coefficients and applied to models with interactions. Consider reverse Kreweras walk with interactions, the functional equation reads,
\begin{align}
\begin{split}
\left(1-t\left(x+y+\frac{1}{xy}\right)\right)Q(x,y)=&\frac{1}{c} + \frac1a\Big(a-1 - \frac{ta}{xy}\Big)Q(x,0) + \frac1b\Big(b - 1 -\frac{tb}{xy}\Big)Q(0,y) \\
&+\left(\frac{1}{abc}(ac+bc-ab-abc)+\frac{t}{xy}\right)Q(0,0)\label{functional reKr}.
\end{split}
\end{align}
The coefficient of $Q(x,0)$ is $\frac1a\Big(a-1 - \frac{ta}{xy}\Big)$. If we regard it as $\frac1a\Big(a-1 - taY_1\Big)$, its conjugate is $\frac1a(a-1-taY_0)$. Similar considerations apply to $\frac1b\Big(b - 1 -\frac{tb}{xy}\Big)$ and the conjugate is $\frac1b\Big(b - 1 -tbX_0\Big)$. Inspired by this, it is straight forward to multiply $(b-1-tbx)(a-1-tay)$ to the functional equation \eqref{functional reKr}. The coefficient of $Q(x,0)$ becomes,
\begin{align}
\begin{split}
    \frac1a\Big(a-1 - \frac{ta}{xy}\Big)(b-1-tbx)(a-1-tay)&=(b-1-tbx)\Big((a-1)^2-ta(a-1)\Big(y+\frac{1}{xy}\Big)+\frac{t^2a^2}{x}\Big)\\
    &\equiv(b-1-tbx)\Big((a-1)^2-ta(a-1)\Big(\frac{1}{t}-x\Big)+\frac{t^2a^2}{x}\Big).\label{reverse Kreweras}
\end{split}
\end{align}
In the second line we apply the equivalent relation $1-t(x+y+\frac{1}{xy})\equiv 0$. The symmetry of $x,y$ shows the coefficient of $Q(0,y)$ is equivalent to $(a-1-tay)\Big((b-1)^2-tb(b-1)\Big(\frac{1}{t}-y\Big)+\frac{t^2b^2}{y}\Big)$. Therefore, after substituting $y\to Y_0$, the coefficient of $Q(x,0)$ is rational of $x$ and the coefficient of $Q(0,Y_0)$ is rational of $Y_0$. We can separate the $[z^>]$ and $[z^<]$ terms (after mapping $z=w(x)$). After multiplying  by $(b-1-tbx)(a-1-tay)$, the term irrelevant to $Q(x,0)$ and $Q(0,y)$ on the right hand-side of \eqref{functional reKr} is a linear function of $x,y,xy,\frac{1}{xy}$. For reverse Kreweras walk, $xy$ and $\frac{1}{xy}$ are both decoupled. Thus, we conclude that the reverse Kreweras walk with interactions has algebraic generating functions. This is exactly the result of \cite{beaton2019quarter}.

Similarly, using the decoupling criterion \cite{raschel2020counting}, one can find the D-algebraic property of the special $9$ models associated with infinite groups stated in \cite{raschel2020counting} via conformal mapping. The obstinate kernel method with combinatorial conformal mapping can be treated as a general form of Tutte's invariant approach, and our approach proves not only the algebraicity, but also the D-finiteness. 

An innocent extension of our approach is that we find what causes the functions to be transcendental. Those models which are not D-algebraic, are transcendental since taking $[z^>]$ terms, or taking contour integral around a branch cut is not a closed operation for D-algebraic functions.

\section{Conclusion and Prospect}
Let us generalize what we have done in this paper. We want to solve an equation,
\begin{align}
   A(x,t)Q(x,t)+B(x,t)F(f(x,t),t)=g(x,t)
\end{align}
where $Q(x,t),F(f(x,t),t)$ are two unknown functions analytic in some domain. From a combinatorial insight, we want to find a conformal mapping $z=w(x)$ such that,
\begin{enumerate}\label{en}
    \item $w^{-1}(z)$ is a Laurent series in $1/z\mathbb{C}[[1/z]][[t]]$;
    \item $f(w^{-1}(z))$ is a Laurent series in $\mathbb{C}[[z]][[t]]$;
\end{enumerate}
or from the analytic insight,
\begin{enumerate}
    \item $w(x)$ maps some subset of analytic domain of $Q(x)$ in $x$-plane to $z$-plane where $Q(w^{-1}(z))$ is convergent except some segment;
    \item $F(f(w^{-1}(z)))$ is analytic near the segment.
\end{enumerate}
Then we can solve $Q(w^{-1}(z)),F(f(w^{-1}(z)))$ by canonical factorization of $A(w^{-1}(z)),B(w^{-1}(z))$ and separating $[z^<],[z^>]$ part. The properties of these functions can be analyzed through the decoupling property of $P_c$, $[z^>],[z^<]$ operators, conformal gluing functions $z=w(x)$ and the properties of $A(w^{-1}(z)),B(w^{-1}(z))$.

The main contribution of this work is that we proved that each step of this approach is valid and can be interpreted both combinatorially and analytically. From the analytic insight of conformal gluing function and index, we see the contact between the kernel method, RBVP and Tutte's invariants method.

Some deep understanding of this conformal mapping is our future research interests. The first condition for $z=w(x)$ is  that it is bi-holomorphic and the second condition demonstrates the gluing property. In \cref{part 2 th 0}, we prove that any conformal gluing function in RBVP can be transformed into a combinatorial conformal mapping by M\"{o}bius transform, and this gives a combinatorial insight for conformal gluing functions. However, besides the simple Joukowsky transform, we are still using the conformal gluing function in RBVP. Can we directly find it through the property: bi-holomorphic and the gluing property. This question is important since in higher dimension or large step walks, we do not have powerful tools like Weierstrass $\wp$ function. Actually, what we need is a differomorphism and conformal is not necessary.
\section*{Acknowledgement}
Ruijie Xu is supported by Yanqi Lake Beijing Institute of Mathematical Sciences and Applications(BIMSA).
\bibliography{reportbib}

\begin{thebibliography}{10}

\bibitem{banderier2002basic}
Cyril Banderier and Philippe Flajolet.
\newblock Basic analytic combinatorics of directed lattice paths.
\newblock {\em Theoretical Computer Science}, 281(1-2):37--80, 2002.

\bibitem{banderier2019}
Cyril Banderier, Christian Krattenthaler, Alan Krinik, Dmitry Kruchinin,
  Vladimir Kruchinin, David Nguyen, and Michael Wallner.
\newblock Explicit formulas for enumeration of lattice paths: basketball and
  the kernel method.
\newblock {\em Lattice Path Combinatorics and Applications}, pages 78--118,
  2019.

\bibitem{beaton2019quarter}
Nicholas~R Beaton, Aleksander~L Owczarek, and Ruijie Xu.
\newblock Quarter-plane lattice paths with interacting boundaries: the
  {K}reweras and reverse {K}reweras models.
\newblock {\em arXiv preprint arXiv:1905.10908, To appear in the Proceedings of
  Transient Transcendence in Transylvania}, 26(3):P3.53, 2019.

\bibitem{beaton2018exact}
NR~Beaton, AL~Owczarek, and A~Rechnitzer.
\newblock Exact solution of some quarter plane walks with interacting
  boundaries.
\newblock {\em The Electronic Journal of Combinatorics}, 26:P3.53, 2019.

\bibitem{bernardi2007bijective}
Olivier Bernardi.
\newblock Bijective counting of {K}reweras walks and loopless triangulations.
\newblock {\em Journal of Combinatorial Theory, Series A}, 114(5):931--956,
  2007.

\bibitem{bostan20163}
Alin Bostan, Mireille Bousquet-M{\'e}lou, Manuel Kauers, and Stephen Melczer.
\newblock On 3-dimensional lattice walks confined to the positive octant.
\newblock {\em Annals of Combinatorics}, 20(4):661--704, 2016.

\bibitem{bostan2017human}
Alin Bostan, Irina Kurkova, and Kilian Raschel.
\newblock A human proof of gessel’s lattice path conjecture.
\newblock {\em Transactions of the American Mathematical Society},
  369(2):1365--1393, 2017.

\bibitem{bousquet2005walks}
Mireille Bousquet-M{\'e}lou.
\newblock Walks in the quarter plane: {K}reweras' algebraic model.
\newblock {\em The Annals of Applied Probability}, 15(2):1451--1491, 2005.

\bibitem{bousquet2016elementary}
Mireille Bousquet-M{\'e}lou.
\newblock An elementary solution of {G}essel's walks in the quadrant.
\newblock {\em Advances in Mathematics}, 303:1171--1189, 2016.

\bibitem{bousquet2010walks}
Mireille Bousquet-M{\'e}lou and Marni Mishna.
\newblock Walks with small steps in the quarter plane.
\newblock {\em Contemporary Mathematics}, 520:1--40, 2010.

\bibitem{brak2005directed}
Richard Brak, Aleksander~L Owczarek, Andrew Rechnitzer, and Stuart~G
  Whittington.
\newblock A directed walk model of a long chain polymer in a slit with
  attractive walls.
\newblock {\em Journal of Physics A: Mathematical and General}, 38(20):4309,
  2005.

\bibitem{buchacher2022orbit}
Manfred Buchacher and Manuel Kauers.
\newblock The orbit-sum method for higher order equations.
\newblock {\em arXiv preprint arXiv:2211.08175}, 2022.

\bibitem{dreyfus2018nature}
Thomas Dreyfus, Charlotte Hardouin, Julien Roques, and Michael~F Singer.
\newblock On the nature of the generating series of walks in the quarter plane.
\newblock {\em Inventiones Mathematicae}, 213(1):139--203, 2018.

\bibitem{fayolle2017random}
Guy Fayolle, Roudolf Iasnogorodski, and Vadim Malyshev.
\newblock {\em Random Walks in the Quarter Plane: Algebraic Methods, Boundary
  Value Problems, Applications to Queueing Systems and Analytic Combinatorics},
  volume~40.
\newblock Springer, 2017.

\bibitem{gessel1980factorization}
Ira~M Gessel.
\newblock A factorization for formal laurent series and lattice path
  enumeration.
\newblock {\em Journal of Combinatorial Theory, Series A}, 28(3):321--337,
  1980.

\bibitem{guy1992lattice}
Richard~K. Guy, C.~Krattenthaler, and Bruce~E. Sagan.
\newblock Lattice paths, reflections, \& dimension-changing bijections.
\newblock {\em Ars Combinatorica}, 34:15, 1992.

\bibitem{kauers2008quasi}
Manuel Kauers and Doron Zeilberger.
\newblock The quasi-holonomic ansatz and restricted lattice walks.
\newblock {\em Journal of Difference Equations and Applications},
  14(10-11):1119--1126, 2008.

\bibitem{kurkova2012functions}
Irina Kurkova and Kilian Raschel.
\newblock On the functions counting walks with small steps in the quarter
  plane.
\newblock {\em Publications Math{\'e}matiques de l'IH{\'E}S}, 116(1):69--114,
  2012.

\bibitem{Litvinchuk2000}
Georgii~S Litvinchuk.
\newblock {\em Solvability theory of boundary value problems and singular
  integral equations with shift}, volume 523.
\newblock Springer Science \& Business Media, 2000.

\bibitem{malyshev1972analytical}
VA~Malyshev.
\newblock An analytical method in the theory of two-dimensional positive random
  walks.
\newblock {\em Siberian Mathematical Journal}, 13(6):917--929, 1972.

\bibitem{melczer2021invitation}
Stephen Melczer.
\newblock {\em An Invitation to Analytic Combinatorics}.
\newblock Springer, 2021.

\bibitem{Melczer2014}
Stephen Melczer and Marni Mishna.
\newblock Singularity analysis via the iterated kernel method.
\newblock {\em Combinatorics, Probability and Computing}, 23(5):861--888, 2014.

\bibitem{mishna2009classifying}
Marni Mishna.
\newblock Classifying lattice walks restricted to the quarter plane.
\newblock {\em Journal of Combinatorial Theory, Series A}, 116(2):460--477,
  2009.

\bibitem{mishna2009two}
Marni Mishna and Andrew Rechnitzer.
\newblock Two non-holonomic lattice walks in the quarter plane.
\newblock {\em Theoretical Computer Science}, 410(38-40):3616--3630, 2009.

\bibitem{omnes1958solution}
R~Omn{\`e}s.
\newblock On the solution of certain singular integral equations of quantum
  field theory.
\newblock {\em II Nuovo Cimento (1955-1965)}, 8(2):316--326, 1958.

\bibitem{postnova2021counting}
Olga Postnova and Dmitry Solovyev.
\newblock Counting filter restricted paths in z2 lattice.
\newblock {\em arXiv preprint arXiv:2107.09774}, 2021.

\bibitem{prodinger2004kernel}
Helmut Prodinger.
\newblock The kernel method: a collection of examples.
\newblock {\em S{\'e}minaire Lotharingien Combinatoire}, 50:B50f, 2004.

\bibitem{raschel2012counting}
Kilian Raschel.
\newblock Counting walks in a quadrant: a unified approach via boundary value
  problems.
\newblock {\em Journal of the European Mathematical Society}, 14(3):749--777,
  2012.

\bibitem{raschel2020counting}
Kilian Raschel, Mireille Bousquet-M{\'e}lou, and Olivier Bernardi.
\newblock Counting quadrant walks via {T}utte's invariant method.
\newblock {\em Discrete Mathematics \& Theoretical Computer Science}, FPSAC
  2016, 2020.

\bibitem{raschel2018walks}
Kilian Raschel and Am{\'e}lie Trotignon.
\newblock On walks avoiding a quadrant.
\newblock {\em arXiv preprint arXiv:1807.08610}, 2018.

\bibitem{Tong2021}
Bin Tong, Olof Salberger, Kun Hao, and Vladimir Korepin.
\newblock Shor--movassagh chain leads to unusual integrable model.
\newblock {\em Journal of Physics A: Mathematical and Theoretical},
  54(39):394002, 2021.

\bibitem{xin2004ring}
Guoce Xin.
\newblock {\em The ring of Malcev-Neumann series and the residue theorem}.
\newblock Brandeis University, 2004.

\bibitem{xu2022interacting}
Ruijie Xu.
\newblock Interacting quarter-plane lattice walk problems: solutions and
  proofs.
\newblock {\em Bulletin of the Australian Mathematical Society},
  105(2):339--340, 2022.

\end{thebibliography}
\end{document}